\documentclass{article}
\usepackage{amsmath}
\usepackage{fancyhdr}
\usepackage{amssymb}
\usepackage[margin=1.5in]{geometry}

\makeatletter

\providecommand{\tabularnewline}{\\}

\usepackage{amsmath}
\usepackage{amsthm}
\theoremstyle{plain}
\newtheorem{thm}{Theorem}[section]
\theoremstyle{definition}
\newtheorem{defn}[thm]{Definition}
\theoremstyle{plain}
\newtheorem{lem}[thm]{Lemma}
\theoremstyle{remark}

\usepackage{kbordermatrix}
\usepackage{listings}
\usepackage{amsmath}
\usepackage{dcolumn}
\usepackage{subfigure}

\DeclareMathOperator{\diag}{diag}

\newcommand{\twocases}[4]{\begin{cases} #1 & \text{if } #2 \\ #3 & \text{if } #4 \end{cases}}

\numberwithin{equation}{section}
\newtheorem{specification}[thm]{Specification}
\theoremstyle{definition}
\newtheorem*{algorithm*}{Algorithm}

\newtheorem{alg}[thm]{Algorithm}
\newtheorem*{acknowledgment*}{Acknowledgment}

\makeatother

\begin{document}

\title{Computing the Complete CS Decomposition}

\author{Brian D. Sutton\thanks{Department of Mathematics, Randolph-Macon College, P.O.~Box 5005, Ashland, VA 23005 USA. email: bsutton@rmc.edu. Supplementary material is available at the author's home page.}}

\maketitle
\maketitle{}

\abstract{An algorithm is developed to compute the complete CS decomposition
(CSD) of a partitioned unitary matrix. Although the existence of the
CSD has been recognized since 1977, prior algorithms compute only
a reduced version (the \mbox{2-by-1} CSD) that is equivalent to two
simultaneous singular value decompositions. The algorithm presented
here computes the complete 2-by-2 CSD, which requires the simultaneous
diagonalization of all four blocks of a unitary matrix partitioned
into a 2-by-2 block structure. The algorithm appears to be the only
fully specified algorithm available. The computation occurs in two
phases. In the first phase, the unitary matrix is reduced to bidiagonal
block form, as described by Sutton and Edelman. In the second phase,
the blocks are simultaneously diagonalized using techniques from bidiagonal
SVD algorithms of Golub, Kahan, and Demmel. The algorithm has a number
of desirable numerical features.

\lhead{}\chead{Computing the Complete CS Decomposition, Brian D. Sutton}\rhead{\thepage}\cfoot{}\renewcommand{\headrulewidth}{0pt}

\section{Introduction}

The \emph{complete CS decomposition} (CSD) applies to any $m$-by-$m$
matrix $X$ from the unitary group $U(m)$, viewed as a 2-by-2 block
matrix,\[X=\kbordermatrix{&q&&m-q\\p&X_{11}&\vrule&X_{12}\\\cline{2-4}m-p&X_{21}&\vrule&X_{22}}.\]For
convenience, we assume $q\leq p$ and $p+q\leq m$. A complete CS
decomposition has the form\begin{gather}
X=\left[\begin{array}{cc}
U_{1}\\
 & U_{2}\end{array}\right]\left[\begin{array}{r|ccc}
C & S & 0 & 0\\
0 & 0 & I_{p-q} & 0\\
\hline -S & C & 0 & 0\\
0 & 0 & 0 & I_{m-p-q}\end{array}\right]\left[\begin{array}{cc}
V_{1}\\
 & V_{2}\end{array}\right]^{*},\label{eq:csd}\\
C=\diag(\cos(\theta_{1}),\dots,\cos(\theta_{q})),\quad S=\diag(\sin(\theta_{1}),\dots,\sin(\theta_{q})),\nonumber \end{gather}
in which $\theta_{1},\dots,\theta_{q}\in[0,\frac{\pi}{2}]$, $U_{1}\in U(p)$,
$U_{2}\in U(m-p)$, $V_{1}\in U(q)$, and $V_{2}\in U(m-q)$. The
letters \emph{CS} in the term \emph{CS decomposition} come from \emph{cosine-sine}.

The major contribution of this paper is an algorithm for computing
(\ref{eq:csd}). We believe this to be the only fully specified algorithm
available for computing the complete CS decomposition. Earlier algorithms
compute only a reduced form, the {}``2-by-1'' CSD, which is defined
in the next section. The algorithm developed in this article is based
on the SVD algorithm of Golub and Kahan and has a number of desirable
numerical properties.

The algorithm proceeds in two phases.

\begin{enumerate}
\item Phase I: Bidiagonalization. In the special case $p=q=\frac{m}{2}$,
the decomposition is \begin{equation}
X=\left[\begin{array}{cc}
P_{1}\\
 & P_{2}\end{array}\right]\left[\begin{array}{cc}
B_{11}^{(0)} & B_{12}^{(0)}\\
B_{21}^{(0)} & B_{22}^{(0)}\end{array}\right]\left[\begin{array}{cc}
Q_{1}\\
 & Q_{2}\end{array}\right]^{*},\label{eq:bdbdecomposition}\end{equation}
in which $B_{11}^{(0)}$ and $B_{21}^{(0)}$ are upper bidiagonal,
$B_{12}^{(0)}$ and $B_{22}^{(0)}$ are lower bidiagonal, and $P_{1}$,
$P_{2}$, $Q_{1}$, and $Q_{2}$ are $q$-by-$q$ unitary. We say
that the middle factor is a real orthogonal matrix in \emph{bidiagonal
block form}. (See Definition \ref{def:bdb}.)
\item Phase II: Diagonalization. The CSD of $\left[\begin{smallmatrix} B^{(0)}_{11} & B^{(0)}_{12} \\ B^{(0)}_{21} & B^{(0)}_{22} \end{smallmatrix}\right]$
is computed, \[
\left[\begin{array}{cc}
B_{11}^{(0)} & B_{12}^{(0)}\\
B_{21}^{(0)} & B_{22}^{(0)}\end{array}\right]=\left[\begin{array}{cc}
U_{1}\\
 & U_{2}\end{array}\right]\left[\begin{array}{rc}
C & S\\
-S & C\end{array}\right]\left[\begin{array}{cc}
V_{1}\\
 & V_{2}\end{array}\right]^{*}.\]

\end{enumerate}
Combining the factorizations gives the CSD of $X$,\begin{equation}
X=\left[\begin{array}{cc}
P_{1}U_{1}\\
 & P_{2}U_{2}\end{array}\right]\left[\begin{array}{rc}
C & S\\
-S & C\end{array}\right]\left[\begin{array}{cc}
Q_{1}V_{1}\\
 & Q_{2}V_{2}\end{array}\right]^{*}.\label{eq:combiningphases}\end{equation}
Phase I is a finite-time procedure first described in \cite{thesis},
and Phase II is an iterative procedure based on ideas from bidiagonal
SVD algorithms \cite{MR1057146,MR0183105}. 

Some of the earliest work related to the CSD was completed by Jordan,
Davis, and Kahan \cite{MR0246155,MR0264450,MR1503705}. The CSD as
we know it today and the term \emph{CS decomposition} first appeared
in a pair of articles by Stewart \cite{MR0461871,MR695598}. Computational
aspects of the 2-by-1 CSD are considered in \cite{bade93,MR857786,MR615522,MR695598,MR796639}
and later articles. A {}``sketch'' of an algorithm for the complete
CSD can be found in a paper by Hari \cite{MR2161439}, but few details
are provided. For general information and more references, see \cite{bai92,MR1417720,MR1287355}.

\subsection{Complete versus 2-by-1 CS decomposition}

Most commonly available CSD algorithms compute what we call the \emph{2-by-1
CS decomposition} of a matrix $\hat{X}$ with orthonormal columns
partitioned into a 2-by-1 block structure. In the special case $p=q=\frac{m}{2}$,
$\hat{X}$ has the form \[\hat{X}=\kbordermatrix{&q\\q&\hat{X}_{11}\\q&\hat{X}_{21}},\]
and the CSD is \[
\hat{X}=\left[\begin{array}{cc}
U_{1}\\
 & U_{2}\end{array}\right]\left[\begin{array}{r}
C\\
-S\end{array}\right]V_{1}^{*}.\]
A naive algorithm for computing the 2-by-1 CSD is to compute two SVD's,
\[
\left\{ \begin{aligned}\hat{X}_{11} & =U_{1}CV_{1}^{*}\\
\hat{X}_{21} & =(-U_{2})SV_{1}^{*},\end{aligned}
\right.\]
reordering rows and columns and adjusting signs as necessary to make
sure that the two occurrences of $V_{1}^{*}$ are identical and that
$C^{2}+S^{2}=I$. This works in theory if no two singular values of
$\hat{X}_{11}$ are repeated, but in practice it works poorly when
there are clustered singular values. Still, the basic idea can form
the backbone of an effective algorithm for the 2-by-1 CSD \cite{MR695598,MR796639}.

Unfortunately, many algorithms for the 2-by-1 CSD do not extend easily
to the complete 2-by-2 CSD. The problem is the more extensive sharing
of singular vectors evident below (still assuming $p=q=\frac{m}{2}$):\begin{equation}
\left\{ \begin{aligned}X_{11} & =U_{1}CV_{1}^{*} & X_{12} & =U_{1}SV_{2}^{*}\\
X_{21} & =(-U_{2})SV_{1}^{*} & X_{22} & =U_{2}CV_{2}^{*}.\end{aligned}
\right.\label{eq:csdsvd}\end{equation}
All four unitary matrices $U_{1}$, $U_{2}$, $V_{1}$, and $V_{2}$
play dual roles, providing singular vectors for two different blocks
of $X$. Enforcing these identities has proven difficult over the
years.

Our algorithm, unlike the naive algorithm, is designed to compute
the four SVD's in (\ref{eq:csdsvd}) \emph{simultaneously}, so that
no discrepancies ever arise.

\subsection{Applications}

Unlike existing 2-by-1 CSD algorithms, the algorithm developed here
fully solves Jordan's problem of angles between linear subspaces of
$\mathbb{R}^{n}$ \cite{MR1503705}. If the columns of matrices $X$
and $Y$ are orthonormal bases for two subspaces of $\mathbb{R}^{n}$,
then the \emph{principal angles} and \emph{principal vectors} between
the subspaces can be computed in terms of the SVD of $X^{T}Y$ \cite{MR1287355}.
The complete CSD, equivalent to four SVD's, simultaneously provides
principal vectors for these subspaces and their orthogonal complements.

In addition, our algorithm can be specialized to compute the 2-by-1
CSD and hence has application to the generalized singular value decomposition.

\subsection{Numerical properties}

The algorithm is designed for numerical stability. All four blocks
of the partitioned unitary matrix are treated simultaneously and with
equal regard, and no cleanup procedure is necessary at the end of
the algorithm. In addition, a new representation for orthogonal matrices
with a certain structure guarantees orthogonality, even on a floating-point
architecture \cite{thesis}.

\subsection{Efficiency}

As with the SVD algorithm of Golub and Kahan, Phase I (bidiagonalization)
is often more expensive than Phase II (diagonalization). For the special
case $p=q=\frac{m}{2}$, bidiagonalization requires about $2m^{3}$
flops---about $\frac{8}{3}q^{3}$ flops to bidiagonalize each block
of $\left[\begin{smallmatrix} X_{11} & X_{12} \\ X_{21} & X_{22} \end{smallmatrix}\right]$
and about $\frac{4}{3}q^{3}$ flops to accumulate each of $P_{1}$,
$P_{2}$, $Q_{1}$, and $Q_{2}$, for a total of about $4\left(\frac{8}{3}q^{3}\right)+4\left(\frac{4}{3}\right)q^{3}=2m^{3}$
flops.

\subsection{Overview of the algorithm}

\subsubsection{Bidiagonal block form}

During Phase I, the input unitary matrix is reduced to \emph{bidiagonal
block form}. A matrix in this form is real orthogonal and has a specific
sign pattern. Bidiagonal block form was independently formulated by
Sutton in 2005 \cite{thesis}. Some similar results appear in a 1993
paper by Watkins \cite{MR1234638}. The matrix structure and a related
decomposition have already been applied to a problem in random matrix
theory by Edelman and Sutton \cite{jacobipaper}.

\begin{defn}
\label{def:bdb}Given $\theta=(\theta_{1},\dots,\theta_{q})\in[0,\frac{\pi}{2}]^{q}$
and $\phi=(\phi_{1},\dots,\phi_{q-1})\in[0,\frac{\pi}{2}]^{q-1}$,
let $c_{i}=\cos\theta_{i}$, $s_{i}=\sin\theta_{i}$, $c'_{i}=\cos\phi_{i}$,
and $s'_{i}=\sin\phi_{i}$. Define $B_{ij}(\theta,\phi)$, $i,j=1,2$,
to be $q$-by-$q$ bidiagonal matrices, as follows. \begin{multline}
\left[\begin{array}{c|c}
B_{11}(\theta,\phi) & B_{12}(\theta,\phi)\\
\hline B_{21}(\theta,\phi) & B_{22}(\theta,\phi)\end{array}\right]\\
=\left[\begin{array}{cccc|cccc}
c_{1} & -s_{1}s'_{1} &  &  & s_{1}c'_{1}\\
 & c_{2}c'_{1} & \ddots &  & c_{2}s'_{1} & s_{2}c'_{2}\\
 &  & \ddots & -s_{q-1}s'_{q-1} &  & \ddots & \ddots\\
 &  &  & c_{q}c'_{q-1} &  &  & c_{q}s'_{q-1} & s_{q}\\
\hline -s_{1} & -c_{1}s'_{1} &  &  & c_{1}c'_{1}\\
 & -s_{2}c'_{1} & \ddots &  & -s_{2}s'_{1} & c_{2}c'_{2}\\
 &  & \ddots & -c_{q-1}s'_{q-1} &  & \ddots & \ddots\\
 &  &  & -s_{q}c'_{q-1} &  &  & -s_{q}s'_{q-1} & c_{q}\end{array}\right].\label{eq:jacobiform}\end{multline}
Any matrix of the form\[
\left[\begin{array}{c|c}
B_{11}(\theta,\phi) & B_{12}(\theta,\phi)\\
\hline B_{21}(\theta,\phi) & B_{22}(\theta,\phi)\end{array}\right]\]
is said to be in \emph{bidiagonal block form} and is necessarily real
orthogonal.
\end{defn}
To clarify (\ref{eq:jacobiform}), the $(q-1,q-1)$ entry of $B_{12}(\theta,\phi)$
is $s_{q-1}c'_{q-1}$, and the $(q-1,q-1)$ entry of $B_{22}(\theta,\phi)$
is $c_{q-1}c'_{q-1}$. Also, if $q=1$, then the matrices are defined
by \[
\left[\begin{array}{c|c}
B_{11}(\theta,\phi) & B_{12}(\theta,\phi)\\
\hline B_{21}(\theta,\phi) & B_{22}(\theta,\phi)\end{array}\right]=\left[\begin{array}{r|r}
c_{1} & s_{1}\\
\hline -s_{1} & c_{1}\end{array}\right].\]

As stated in the definition, any matrix whose entries satisfy the
relations of (\ref{eq:jacobiform}) is necessarily real orthogonal.
The reverse is true as well---any orthogonal matrix $X$ with the
bidiagonal structure and sign pattern of (\ref{eq:jacobiform}) is
expressible in terms of some $\theta$ and $\phi$. (This is implicit
in \cite{jacobipaper,thesis}.) Furthermore, every unitary matrix
is equivalent to a matrix in bidiagonal block form, as stated in the
next theorem.

\begin{thm}
\label{thm:bdbexists}Given any $m$-by-$m$ unitary matrix $X$ and
integers $p$, $q$ such that $0\leq q\leq p$ and $p+q\leq m$, there
exist matrices $P_{1}\in U(p)$, $P_{2}\in U(m-p)$, $Q_{1}\in U(q)$,
and $Q_{2}\in U(m-q)$ such that\[
X=\left[\begin{array}{cc}
P_{1}\\
 & P_{2}\end{array}\right]\left[\begin{array}{c|ccc}
B_{11}(\theta,\phi) & B_{12}(\theta,\phi) & 0 & 0\\
0 & 0 & I_{p-q} & 0\\
\hline B_{21}(\theta,\phi) & B_{22}(\theta,\phi) & 0 & 0\\
0 & 0 & 0 & I_{m-p-q}\end{array}\right]\left[\begin{array}{cc}
Q_{1}\\
 & Q_{2}\end{array}\right]^{*}\]
for some $\theta=(\theta_{1},\dots,\theta_{q})\in[0,\frac{\pi}{2}]^{q}$
and $\phi=(\phi_{1},\dots,\phi_{q-1})\in[0,\frac{\pi}{2}]^{q-1}$. 
\end{thm}
A proof of the theorem has already been published in \cite{jacobipaper,thesis},
along with an algorithm for computing the decomposition. The algorithm
applies pairs of Householder reflectors to the left and right of $X$,
causing the structure to evolve as in Fig.~\ref{fig:bidiagonalizepattern}.
This serves as Phase I of the CSD algorithm.

\begin{figure}
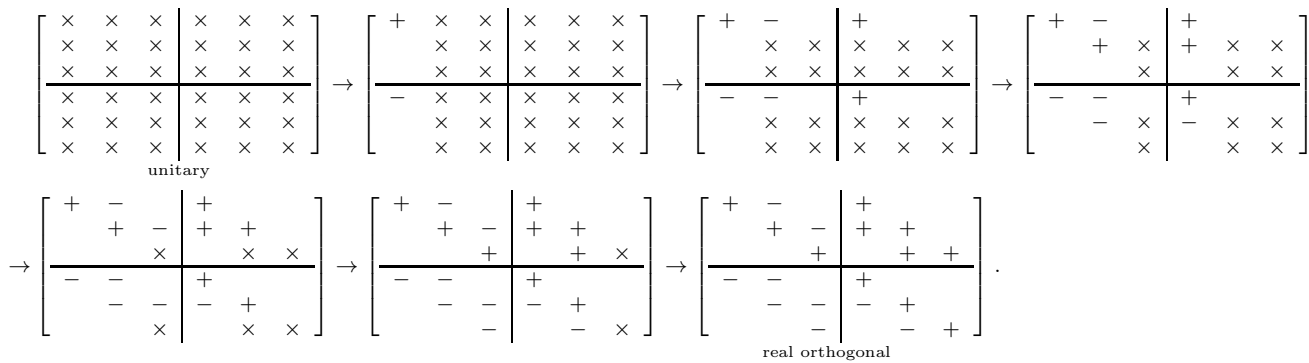

\footnotesize
\hspace*{-1.05in}\begin{minipage}{7.5in}
\begin{align*}
 & \underset{\text{unitary}}{\left[\begin{array}{ccc|ccc}
\times & \times & \times & \times & \times & \times\\
\times & \times & \times & \times & \times & \times\\
\times & \times & \times & \times & \times & \times\\
\hline \times & \times & \times & \times & \times & \times\\
\times & \times & \times & \times & \times & \times\\
\times & \times & \times & \times & \times & \times\end{array}\right]}\rightarrow\left[\begin{array}{ccc|ccc}
+ & \times & \times & \times & \times & \times\\
 & \times & \times & \times & \times & \times\\
 & \times & \times & \times & \times & \times\\
\hline - & \times & \times & \times & \times & \times\\
 & \times & \times & \times & \times & \times\\
 & \times & \times & \times & \times & \times\end{array}\right]\rightarrow\left[\begin{array}{ccc|ccc}
+ & - &  & +\\
 & \times & \times & \times & \times & \times\\
 & \times & \times & \times & \times & \times\\
\hline - & - &  & +\\
 & \times & \times & \times & \times & \times\\
 & \times & \times & \times & \times & \times\end{array}\right]\rightarrow\left[\begin{array}{ccc|ccc}
+ & - &  & +\\
 & + & \times & + & \times & \times\\
 &  & \times &  & \times & \times\\
\hline - & - &  & +\\
 & - & \times & - & \times & \times\\
 &  & \times &  & \times & \times\end{array}\right]\\
\rightarrow & \left[\begin{array}{ccc|ccc}
+ & - &  & +\\
 & + & - & + & +\\
 &  & \times &  & \times & \times\\
\hline - & - &  & +\\
 & - & - & - & +\\
 &  & \times &  & \times & \times\end{array}\right]\rightarrow\left[\begin{array}{ccc|ccc}
+ & - &  & +\\
 & + & - & + & +\\
 &  & + &  & + & \times\\
\hline - & - &  & +\\
 & - & - & - & +\\
 &  & - &  & - & \times\end{array}\right]\underset{\text{real orthogonal}}{\rightarrow\left[\begin{array}{ccc|ccc}
+ & - &  & +\\
 & + & - & + & +\\
 &  & + &  & + & +\\
\hline - & - &  & +\\
 & - & - & - & +\\
 &  & - &  & - & +\end{array}\right]}.\end{align*}
\end{minipage}

\caption{\label{fig:bidiagonalizepattern}Reduction to bidiagonal block form}

\end{figure}

\subsubsection{\label{sub:implicitqr}Simultaneous SVD steps}

Phase II of the algorithm simultaneously applies the bidiagonal SVD
algorithm of Golub and Kahan \cite{MR1057146,MR0183105,MR1417720}
to each of the four blocks of a matrix in bidiagonal block form.

The bidiagonal SVD algorithm is an iterative scheme. Given an initial
bidiagonal matrix $B^{(0)}$, the algorithm produces a sequence $B^{(0)}\rightarrow B^{(1)}\rightarrow B^{(2)}\rightarrow\cdots\rightarrow\Sigma$
converging to a diagonal matrix of singular values. Implicitly, the
step from $B^{(n)}$ to $B^{(n+1)}$ involves a QR factorization of
$(B^{(n)})^{T}B^{(n)}-\sigma^{2}I$, for some appropriately chosen
$\sigma\geq0$, but in practice the matrix $(B^{(n)})^{T}B^{(n)}$
is never explicitly formed. Instead, the transformation from $B^{(n)}$
to $B^{(n+1)}$ is accomplished through a sequence of Givens rotations.
The first Givens rotation introduces a {}``bulge,''$ $ and the
subsequent rotations {}``chase the bulge'' away.

Our algorithm applies this idea simultaneously to all four blocks
to execute a \emph{CSD step}. First, two bulges are introduced by
a Givens rotation (Fig.~\ref{fig:implicitqrpatterna}), and then
the bulges are chased away, also by Givens rotations (Fig.~\ref{fig:implicitqrpatternb}).
The end result is a new matrix in bidiagonal block form whose blocks
tend to be closer to diagonal than the original blocks.

\begin{figure}
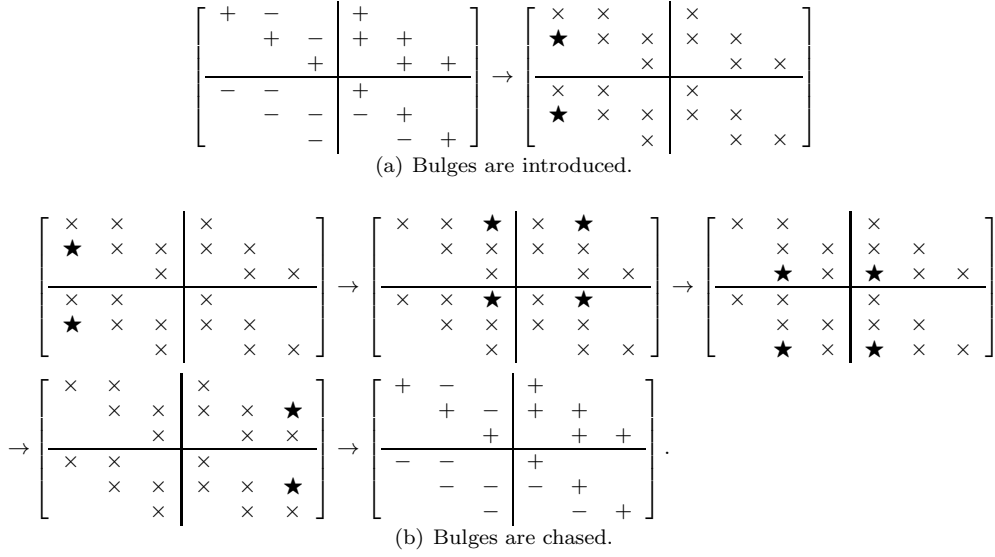

\footnotesize
\subfigure[\label{fig:implicitqrpatterna}Bulges are introduced.]{%
\begin{minipage}[t][1\totalheight]{1\columnwidth}%
\[
\left[\begin{array}{ccc|ccc}
+ & - &  & +\\
 & + & - & + & +\\
 &  & + &  & + & +\\
\hline - & - &  & +\\
 & - & - & - & +\\
 &  & - &  & - & +\end{array}\right]\rightarrow\left[\begin{array}{ccc|ccc}
\times & \times &  & \times\\
\bigstar & \times & \times & \times & \times\\
 &  & \times &  & \times & \times\\
\hline \times & \times &  & \times\\
\bigstar & \times & \times & \times & \times\\
 &  & \times &  & \times & \times\end{array}\right]\]
\end{minipage}}

\subfigure[\label{fig:implicitqrpatternb}Bulges are chased.]{%
\begin{minipage}[t][1\totalheight]{1\columnwidth}%
\begin{align*}
 & \left[\begin{array}{ccc|ccc}
\times & \times &  & \times\\
\bigstar & \times & \times & \times & \times\\
 &  & \times &  & \times & \times\\
\hline \times & \times &  & \times\\
\bigstar & \times & \times & \times & \times\\
 &  & \times &  & \times & \times\end{array}\right]\rightarrow\left[\begin{array}{ccc|ccc}
\times & \times & \bigstar & \times & \bigstar\\
 & \times & \times & \times & \times\\
 &  & \times &  & \times & \times\\
\hline \times & \times & \bigstar & \times & \bigstar\\
 & \times & \times & \times & \times\\
 &  & \times &  & \times & \times\end{array}\right]\rightarrow\left[\begin{array}{ccc|ccc}
\times & \times &  & \times\\
 & \times & \times & \times & \times\\
 & \bigstar & \times & \bigstar & \times & \times\\
\hline \times & \times &  & \times\\
 & \times & \times & \times & \times\\
 & \bigstar & \times & \bigstar & \times & \times\end{array}\right]\\
\rightarrow & \left[\begin{array}{ccc|ccc}
\times & \times &  & \times\\
 & \times & \times & \times & \times & \bigstar\\
 &  & \times &  & \times & \times\\
\hline \times & \times &  & \times\\
 & \times & \times & \times & \times & \bigstar\\
 &  & \times &  & \times & \times\end{array}\right]\rightarrow\left[\begin{array}{ccc|ccc}
+ & - &  & +\\
 & + & - & + & +\\
 &  & + &  & + & +\\
\hline - & - &  & +\\
 & - & - & - & +\\
 &  & - &  & - & +\end{array}\right].\end{align*}
\end{minipage}}

\caption{\label{fig:implicitqrpattern}CSD step}

\end{figure}

\subsubsection{The driver routine}

The algorithm as a whole proceeds roughly as follows.

\begin{itemize}
\item Execute Algorithm $\mathbf{bidiagonalize}$ to transform $X$ to bidiagonal
block form. (See Fig.~\ref{fig:bidiagonalizepattern}).
\item Until convergence,

\begin{itemize}
\item Execute Algorithm $\mathbf{csd\_step}$ to apply four simultaneous
SVD steps. (See Fig.~\ref{fig:implicitqrpattern}.)
\end{itemize}
\end{itemize}
The algorithm as a whole is represented by Fig.~\ref{fig:outline}.

\begin{figure}
\small
\subfigure[Reduction to bidiagonal block form]{%
\begin{minipage}[t][1\totalheight]{1\columnwidth}%
\[
X=\left[\begin{array}{cc}
X_{11} & X_{12}\\
X_{21} & X_{22}\end{array}\right]\rightarrow\left[\begin{array}{cc}
B_{11}^{(0)} & B_{12}^{(0)}\\
B_{21}^{(0)} & B_{22}^{(0)}\end{array}\right]\]
\end{minipage}}

\subfigure[\label{fig:outlineb}Iteration of the CSD step]{%
\begin{minipage}[t][1\totalheight]{1\columnwidth}%
\[
\left[\begin{array}{cc}
B_{11}^{(0)} & B_{12}^{(0)}\\
B_{21}^{(0)} & B_{22}^{(0)}\end{array}\right]\rightarrow\left[\begin{array}{cc}
B_{11}^{(1)} & B_{12}^{(1)}\\
B_{21}^{(1)} & B_{22}^{(1)}\end{array}\right]\rightarrow\left[\begin{array}{cc}
B_{11}^{(2)} & B_{12}^{(2)}\\
B_{21}^{(2)} & B_{22}^{(2)}\end{array}\right]\rightarrow\cdots\rightarrow\left[\begin{array}{cc}
B_{11}^{(N)} & B_{12}^{(N)}\\
B_{21}^{(N)} & B_{22}^{(N)}\end{array}\right]=\left[\begin{array}{cc}
C & S\\
-S & C\end{array}\right]\]
\end{minipage}}

\caption{\label{fig:outline}Computing the complete CSD}

\end{figure}

Matrices in bidiagonal block form may be represented implicitly in
terms of $\theta$ and $\phi$ from Definition \ref{def:bdb}. In
fact, the overall algorithm implicitly represents the sequence of
Fig.~\ref{fig:outlineb} as\[
(\theta^{(0)},\phi^{(0)})\rightarrow(\theta^{(1)},\phi^{(1)})\rightarrow(\theta^{(2)},\phi^{(2)})\rightarrow\cdots\rightarrow(\theta^{(N)},\phi^{(N)}).\]
The implicitly represented matrices are exactly orthogonal, even in
floating-point. The process stops when $\phi^{(N)}$ is sufficiently
close to $(0,\dots,0)$; then the blocks of (\ref{eq:jacobiform})
are diagonal up to machine precision.

\subsection{Overview of the article}

The remainder of the article is organized as follows.

\begin{center}
\begin{tabular}{rl}
\hline 
Section & Title\tabularnewline
\hline
2 & Phase I: Algorithm $\mathbf{bidiagonalize}$\tabularnewline
3 & Reviewing and extending the SVD step\tabularnewline
4 & Phase II: Algorithm $\mathbf{csd\_step}$\tabularnewline
5 & Algorithm $\mathbf{csd}$\tabularnewline
6 & On numerical stability\tabularnewline
\hline
\end{tabular}
\par\end{center}

\noindent{}The final section contains results of numerical tests
on a BLAS/LAPACK-based implementation, which is available from the
author's web site.

\section{\label{sec:bidiagonalize}Phase I: Algorithm $\mathbf{bidiagonalize}$}

Phase I of the CSD algorithm is to transform the partitioned unitary
matrix $X$ to bidiagonal block form.

\begin{specification}\label{spec:bidiagonalize}Given an $m$-by-$m$
unitary matrix $X$ and integers $p$, $q$ with $0\leq q\leq p$
and $p+q\leq m$, $\mathbf{bidiagonalize}(X,p,q)$ should compute
$\theta^{(0)}=(\theta_{1}^{(0)},\dots,\theta_{q}^{(0)})\in[0,\frac{\pi}{2}]^{q}$,
$\phi^{(0)}=(\phi_{1}^{(0)},\dots,\phi_{q-1}^{(0)})\in[0,\frac{\pi}{2}]^{q-1}$,
$P_{1}\in U(p)$, $P_{2}\in U(m-p)$, $Q_{1}\in U(q)$, and $Q_{2}\in U(m-q)$
such that\begin{equation}
X=\left[\begin{array}{cc}
P_{1}\\
 & P_{2}\end{array}\right]\left[\begin{array}{c|ccc}
B_{11}^{(0)} & B_{12}^{(0)} & 0 & 0\\
0 & 0 & I_{p-q} & 0\\
\hline B_{21}^{(0)} & B_{22}^{(0)} & 0 & 0\\
0 & 0 & 0 & I_{m-p-q}\end{array}\right]\left[\begin{array}{cc}
Q_{1}\\
 & Q_{2}\end{array}\right]^{*},\label{eq:bidiagonalizeresult}\end{equation}
in which $B_{ij}^{(0)}=B_{ij}(\theta^{(0)},\phi^{(0)})$, $i,j=1,2$,
are bidiagonal matrices defined in (\ref{eq:jacobiform}).\end{specification}

The algorithm has already appeared in \cite{jacobipaper,thesis}.
It is reproduced here. Matlab-style indexing is used---$A(i,j)$ refers
to the $i,j$ entry of $A$; $A(i:k,j:l)$ refers to the submatrix
of $A$ in rows $i,\dots,k$ and columns $j,\dots,l$; $A(i:k,:)$
refers to the submatrix of $A$ in rows $i,\dots,k$; and so on. Also,
$\mathbf{house}(x)$ constructs a Householder reflector $F=\omega(I-\beta vv^{*})$
for which the first entry of $Fx$ is real and nonnegative and the
remaining entries are zero. (This is an abuse of common terminology---$F$
is not Hermitian if $\omega$ is not real.) If given the empty vector
$()$, $\mathbf{house}$ returns an identity matrix. Finally, $c_{i}$,
$s_{i}$, $c'_{i}$, and $s'_{i}$ are shorthand for $\cos\theta_{i}^{(0)}$,
$\sin\theta_{i}^{(0)}$, $\cos\phi_{i}^{(0)}$, and $\sin\phi_{i}^{(0)}$,
respectively.

\begin{alg}[$\mathbf{bidiagonalize}$]\label{alg:bidiagonalize}
\ 
\lstset{columns=flexible,escapechar=\%,mathescape,numberblanklines=false,numbers=left,numberstyle=\tiny}
\begin{lstlisting}
$Y := X$;

for $i := 1,\dots,q$

    $u_1 := \twocases{Y(1:p,1)}{i=1}{c'_{i-1}Y(i:p,i)+s'_{i-1}Y(i:p,q-1+i)}{i>1;}$
    $u_2 := \twocases{-Y(p+1:m,1)}{i=1}{-c'_{i-1}Y(p+i:m,i)-s'_{i-1}Y(p+i:m,q-1+i)}{i>1;}$
    $\theta^{(0)}_i := \mathbf{atan2}(\|u_2\|,\|u_1\|)$;
    $P_1^{(i)} := \left[\begin{smallmatrix} I_{i-1} & \\ & \mathbf{house}(u_1)^*\end{smallmatrix}\right];\;\;P_2^{(i)} := \left[\begin{smallmatrix} I_{i-1} & \\ & \mathbf{house}(u_2)^*\end{smallmatrix}\right]$;
    $Y := \left[\begin{smallmatrix}P_1^{(i)} & \\ & P_2^{(i)}\end{smallmatrix}\right]^*Y$;

    $v_1 := \twocases{-s_{i}Y(i,i+1:q)-c_{i}Y(p+i,i+1:q)}{i<q}{()}{i=q;}$
    $v_2 := s_{i}Y(i,q+i:m)+c_{i}Y(p+i,q+i:m)$;
    if $i<q$, then
        $\phi^{(0)}_i := \mathbf{atan2}(\|v_1\|,\|v_2\|)$;
        $Q_1^{(i)} := \left[\begin{smallmatrix} I_{i} & \\ & \mathbf{house}(v_1^*)^*\end{smallmatrix}\right]$;
    else
        $Q_1^{(q)} := I_q$;
    end if
    $Q_2^{(i)} := \left[\begin{smallmatrix}I_{i-1} & \\ & \mathbf{house}(v_2^*)^*\end{smallmatrix}\right]$;
    $Y := Y\left[\begin{smallmatrix}Q_1^{(i)}&\\&Q_2^{(i)}\end{smallmatrix}\right]$;

end for

$P_1:=P_1^{(1)}\cdots P_1^{(q)}$; $P_2:=P_2^{(1)}\cdots P_2^{(q)}$; $Q_1:=Q_1^{(1)}\cdots Q_1^{(q-1)}$; $Q_2:=Q_2^{(1)}\cdots Q_2^{(q)}$;

%Using an LQ factorization, diagonalize the submatrix of $Y$ lying in rows $q+1,\dots,p$ and $p+q+1,\dots,m$ and columns $2q+1,\dots,m$ and update $Q_2$;%
\end{lstlisting}
\end{alg}

\begin{thm}
Algorithm $\mathbf{bidiagonalize}$ satisfies Specification \ref{spec:bidiagonalize}.
\end{thm}
The proof is available in \cite{jacobipaper,thesis}.

We illustrate the algorithm by concentrating on the case $m=6$, $p=3$,
$q=3$.

In the beginning, the matrix entries can have any signs,

{\footnotesize \[
Y:=X=\left[\begin{array}{ccc|ccc}
\times & \times & \times & \times & \times & \times\\
\times & \times & \times & \times & \times & \times\\
\times & \times & \times & \times & \times & \times\\
\hline \times & \times & \times & \times & \times & \times\\
\times & \times & \times & \times & \times & \times\\
\times & \times & \times & \times & \times & \times\end{array}\right].\]
}{\footnotesize \par}

To introduce zeros into column 1, two Householder reflectors, based
on $u_{1}=Y(1\!:\!3,1)$ and $u_{2}=-Y(4\!:\!6,1)$, respectively,
are applied.

{\footnotesize \[
Y:=\left[\begin{array}{c|c}
\mathbf{house}(u_{1})\\
\hline  & \mathbf{house}(u_{2})\end{array}\right]Y=\left[\begin{array}{ccc|ccc}
+ & \times & \times & \times & \times & \times\\
0 & \times & \times & \times & \times & \times\\
0 & \times & \times & \times & \times & \times\\
\hline - & \times & \times & \times & \times & \times\\
0 & \times & \times & \times & \times & \times\\
0 & \times & \times & \times & \times & \times\end{array}\right].\]
}In fact, because $Y$ is unitary, $\|u_{1}\|^{2}+\|u_{2}\|^{2}=1$,
and $\theta_{1}^{(0)}$ is defined so that

{\footnotesize \[
Y=\left[\begin{array}{ccc|ccc}
c_{1} & \times & \times & \times & \times & \times\\
0 & \times & \times & \times & \times & \times\\
0 & \times & \times & \times & \times & \times\\
\hline -s_{1} & \times & \times & \times & \times & \times\\
0 & \times & \times & \times & \times & \times\\
0 & \times & \times & \times & \times & \times\end{array}\right].\]
}{\footnotesize \par}

Next, the algorithm focuses on rows 1 and 4. Now something nonobvious
happens. $Y(1,2\!:\!3)$ and $Y(4,2\!:\!3)$ are colinear (because
$Y(:,1)$ is orthogonal to $Y(:,2)$ and $Y(:,3)$), as are $Y(1,4\!:\!6)$
and $Y(4,4\!:\!6)$. (By \emph{colinear}, we mean that the absolute
value of the inner product equals the product of norms.) The row vectors
$v_{1}$ and $v_{2}$ are defined so that $v_{1}$ is colinear with
$Y(1,2\!:\!3)$ and $Y(4,2\!:\!3)$ and $v_{2}$ is colinear with
$Y(1,4\!:\!6)$ and $Y(4,4\!:\!6)$. Computing $\phi_{1}^{(0)}$ and
multiplying by Householder reflectors gives

{\footnotesize \[
Y:=Y\left[\begin{array}{cc|c}
1 &  & {}\\
 & \mathbf{house}(v_{1}^{*}) & {}\\
\hline  &  & \mathbf{house}(v_{2}^{*})\end{array}\right]=\left[\begin{array}{ccc|ccc}
c_{1} & -s_{1}s'_{1} & 0 & s_{1}c'_{1} & 0 & 0\\
0 & \times & \times & \times & \times & \times\\
0 & \times & \times & \times & \times & \times\\
\hline -s_{1} & -c_{1}s'_{1} & 0 & c_{1}c'_{1} & 0 & 0\\
0 & \times & \times & \times & \times & \times\\
0 & \times & \times & \times & \times & \times\end{array}\right].\]
}{\footnotesize \par}

The algorithm proceeds in a similar fashion. Now, $Y(2\!:\!3,2)$
and $Y(2\!:\!3,4)$ are colinear, as are $Y(5\!:\!6,2)$ and $Y(5\!:\!6,4)$.
By computing $\theta_{2}^{(0)}$ and applying Householder reflectors,
we obtain

{\footnotesize \[
Y:=\left[\begin{array}{ccc|ccc}
c_{1} & -s_{1}s'_{1} & 0 & s_{1}c'_{1} & 0 & 0\\
0 & c_{2}c'_{1} & \times & c_{2}s'_{1} & \times & \times\\
0 & 0 & \times & 0 & \times & \times\\
\hline -s_{1} & -c_{1}s'_{1} & 0 & c_{1}c'_{1} & 0 & 0\\
0 & -s_{2}c'_{1} & \times & -s_{2}s'_{1} & \times & \times\\
0 & 0 & \times & 0 & \times & \times\end{array}\right],\]
}then another pair of Householder reflectors gives

{\footnotesize \[
Y:=\left[\begin{array}{ccc|ccc}
c_{1} & -s_{1}s'_{1} & 0 & s_{1}c'_{1} & 0 & 0\\
0 & c_{2}c'_{1} & -s_{2}s'_{2} & c_{2}s'_{1} & s_{2}c'_{2} & 0\\
0 & 0 & \times & 0 & \times & \times\\
\hline -s_{1} & -c_{1}s'_{1} & 0 & c_{1}c'_{1} & 0 & 0\\
0 & -s_{2}c'_{1} & -c_{2}s'_{2} & -s_{2}s'_{1} & c_{2}c'_{2} & 0\\
0 & 0 & \times & 0 & \times & \times\end{array}\right],\]
}and so on. Note that the final matrix is represented implicitly by
$\theta_{1}^{(0)},\dots,\theta_{q}^{(0)}$ and $\phi_{1}^{(0)},\dots,\phi_{q-1}^{(0)}$,
so that it is exactly orthogonal, even on a floating-point architecture.

The following theorem will be useful later. A \emph{diagonal signature
matrix} is a diagonal matrix whose diagonal entries are $\pm1$.

\begin{thm}
\label{thm:signsarenothingtofear}If $B_{11}$ and $B_{21}$ are upper
bidiagonal, $B_{12}$ and $B_{22}$ are lower bidiagonal, and\[
\left[\begin{array}{cc}
B_{11} & B_{12}\\
B_{21} & B_{22}\end{array}\right]\]
is real orthogonal (but not necessarily having the sign pattern required
for bidiagonal block form), then there exist diagonal signature matrices
$D_{1}$, $D_{2}$, $E_{1}$, $E_{2}$ such that\[
\left[\begin{array}{cc}
D_{1}\\
 & D_{2}\end{array}\right]\left[\begin{array}{cc}
B_{11} & B_{12}\\
B_{21} & B_{22}\end{array}\right]\left[\begin{array}{cc}
E_{1}\\
 & E_{2}\end{array}\right]\]
 is a real orthogonal matrix in bidiagonal block form.
\end{thm}
\begin{proof}
Run $\mathbf{bidiagonalize}$. $P_{1}$, $P_{2}$, $Q_{1}$, and $Q_{2}$
are products of Householder reflectors that are in fact diagonal signature
matrices because the input matrix already has the correct zero/nonzero
pattern. Let $D_{1}=P_{1}$, $D_{2}=P_{2}$, $E_{1}=Q_{1}$, and $E_{2}=Q_{2}$.
$\qed$
\end{proof}

\section{\label{sec:svdstep}Reviewing and extending the SVD step}

In \cite{MR0183105,MR1553974}, Golub, Kahan, and Reinsch developed
a bulge-chasing method for computing the SVD of a bidiagonal matrix.
This method, as modified by Demmel and Kahan \cite{MR1057146}, is
implemented as one of the most heavily used SVD routines in LAPACK.
Phase II of our CSD algorithm is based on this SVD method.

Given a real upper bidiagonal matrix $B$ and a shift $\sigma\geq0$,
the SVD step of Golub and Kahan applies a unitary equivalence $\bar{B}=S^{T}BT$
derived from QR iteration on $B^{T}B-\sigma^{2}I$. The step is designed
so that iteration will drive a superdiagonal entry of $\bar{B}$ to
zero (especially quickly if the shifts lie near singular values of
$B$).

This section reviews the SVD step of Golub and Kahan. There are two
notable aspects not present in most descriptions.

\begin{itemize}
\item First, a certain left-right symmetry is emphasized---not only is the
SVD step equivalent to a QR step on $B^{T}B-\sigma^{2}I$; it is also
equivalent to a QR step on $BB^{T}-\sigma^{2}I$.
\item Second, the SVD step is extended to handle any number of zeros on
the bidiagonal band of $B$. (The original SVD step of Golub and Kahan
requires special handling as soon as a zero appears. The modification
by Demmel and Kahan relaxes this, but only for zeros on the main diagonal
and only when the shift is $\sigma=0$.) The modification is necessary
in Section~\ref{sec:qrstep}.
\end{itemize}

\subsection{QR step}

The SVD iteration of Golub and Kahan is derived from QR iteration
for symmetric tridiagonal matrices.

Given a real symmetric tridiagonal matrix $A$ and a shift $\lambda\in\mathbb{R}$,
a single QR step is accomplished by computing a QR factorization\begin{equation}
A-\lambda I=QR,\label{eq:qrtridiagonal1}\end{equation}
then reversing the factors and putting $\lambda I$ back,\begin{equation}
\bar{A}:=RQ+\lambda I.\label{eq:qrtridiagonal2}\end{equation}
Note that the resulting $\bar{A}$ is symmetric tridiagonal, because
$RQ=Q^{T}(A-\lambda I)Q$ is upper Hessenberg and symmetric.

Note that if $A$ is \emph{unreduced}, i.e., its subdiagonal entries
are all nonzero, then $Q$ and $R$ are unique up to signs. (Specifically,
every QR factorization is of the form $A-\lambda I=(QD)(DR)$ for
some diagonal signature matrix $D$.) However, if $A$ has zero entries
on its subdiagonal, then making the QR factorization unique requires
extra care. The following definition introduces a {}``preferred''
QR factorization. There are two important points: (1) the existence
of the forthcoming CSD step relies on the uniqueness of the preferred
QR factorization, and (2) the handling of noninvertible $A$ supports
the CSD deflation procedure. Below, the notation $A\oplus B$ refers
to the block diagonal matrix $\left[\begin{smallmatrix}A&\\&B\end{smallmatrix}\right]$.

\begin{defn}
\label{def:preferredqr}Let $A$ be an $m$-by-$m$ real symmetric
tridiagonal matrix. Express $A$ as \[
A=\left[\begin{array}{cccc}
A_{1}\\
 & A_{2}\\
 &  & \ddots\\
 &  &  & A_{r}\end{array}\right],\]
in which each $A_{i}$ is either an unreduced tridiagonal matrix or
$1$-by-$1$. A \emph{preferred QR factorization }for $A$ is a QR
factorization with a special form. There are two cases.
\begin{description}
\item [{Case~1:~$A$~is~invertible.}] Then a preferred QR factorization
has the form

\[
A=\left[\begin{array}{cccc}
Q_{1}\\
 & Q_{2}\\
 &  & \ddots\\
 &  &  & Q_{r}\end{array}\right]\left[\begin{array}{cccc}
R_{1}\\
 & R_{2}\\
 &  & \ddots\\
 &  &  & R_{r}\end{array}\right]\]
and satisfies the following conditions:

\begin{enumerate}
\item every $Q_{i}$ has the same number of rows and columns as $A_{i}$,
is real orthogonal and upper Hessenberg, and has positive subdiagonal
and determinant one (unless it is 1-by-1, in which case it equals
the scalar 1), and
\item every $R_{i}$ has the same number of rows and columns as $A_{i}$
and is upper triangular.
\end{enumerate}
\item [{Case~2:~$A$~is~not~invertible.}] Define $Q_{1},\dots,Q_{r}$
and $R_{1},\dots,R_{r}$ as in the first case, let $k$ be the least
index identifying a noninvertible $A_{k}$, and let $l$ be the index
of this block's last row and column in the overall matrix $A$. (Note
that the first zero diagonal entry of $R_{1}\oplus\cdots\oplus R_{r}$
must be in position $(l,l)$.) Also, let\[
P=\left[\begin{array}{ccc}
I_{l-1} & 0 & 0\\
0 & 0 & \pm1\\
0 & I_{m-l} & 0\end{array}\right],\]
with the sign chosen so that $\det(P)=1$. A preferred QR factorization
for $A$ is $A=QR$ with $Q=(Q_{1}\oplus\cdots\oplus Q_{k}\oplus I_{m-l})P$
and $R=P^{T}(R_{1}\oplus\cdots\oplus R_{k}\oplus A_{k+1}\oplus\cdots\oplus A_{r})$.
\end{description}
\end{defn}
\begin{thm}
The terminology is valid: every {}``preferred QR factorization''
really is a QR factorization.
\end{thm}
{}

\begin{thm}
If $A$ is a real symmetric tridiagonal matrix, then a preferred QR
factorization $A=QR$ exists and is unique. 
\end{thm}
The proofs are straightforward.

\begin{thm}
\label{thm:qrstepperfectshift}Apply the QR step (\ref{eq:qrtridiagonal1})--(\ref{eq:qrtridiagonal2})
to a real symmetric tridiagonal matrix $A$ using the preferred QR
factorization. If the shift $\lambda$ is an eigenvalue of $A$, then
deflation occurs immediately: the resulting $\bar{A}$ has the form\[
\bar{A}=\left[\begin{array}{c|c}
\mbox{\huge{$*$}}\\
\hline  & \lambda\end{array}\right].\]

\end{thm}
\begin{proof}
$A-\lambda I$ is not invertible, so its preferred QR factorization
satisfies Case 2 of Definition \ref{def:preferredqr}. The rest of
the proof uses the notation of that definition. The last row of $R_{k}$
contains all zeros, so row $l$ of $R_{1}\oplus\cdots\oplus R_{k}\oplus(A_{k+1}-\lambda I)\oplus\cdots\oplus(A_{r}-\lambda I)$
contains all zeros. The permutation matrix $P$ is designed to slide
this row of all zeros to the bottom of $R$, and hence the last row
of $RQ$ also contains all zeros. By symmetry, the last column of
$RQ$ also contains only zeros, and so $RQ$ is block diagonal with
a 1-by-1 block equaling zero in the bottom-right corner. Adding $\lambda I$
to produce $\bar{A}=RQ+\lambda I$ makes the bottom-right entry equal
$\lambda$. $\qed$
\end{proof}
The orthogonal factor $Q$ in a preferred QR factorization can be
expressed as a product of Givens rotations in a particularly simple
way. A 2-by-2 \emph{Givens rotation} with an angle of $\theta$ is
a matrix $\left[\begin{smallmatrix} c & -s \\ s & c \end{smallmatrix}\right]$
in which $c=\cos\theta$ and $s=\sin\theta$. More generally, an $m$-by-$m$
\emph{Givens rotation} is an $m$-by-$m$ real orthogonal matrix with
a 2-by-2 principal submatrix that is a Givens rotation and whose principal
submatrix lying in the other rows and columns is the $(m-2)$-by-$(m-2)$
identity matrix.

\begin{thm}
\label{thm:productofgivens}The orthogonal factor $Q$ in a preferred
QR factorization of an $m$-by-$m$ real symmetric tridiagonal matrix
can be expressed uniquely as a product of Givens rotations $G_{1}\cdots G_{m-1}$
in which $G_{j}$ rotates columns $j$ and $j+1$ through an angle
in $[0,\pi)$.
\end{thm}
\begin{proof}
Existence is proved by Algorithm \ref{alg:implicitqrtridiagonal}
below. Uniqueness is guaranteed by the upper Hessenberg structure:
inductively, the $(j+1,j)$ entry of $Q$ determines $G_{j}$. $\qed$
\end{proof}
\begin{alg}[$\mathbf{givens}$] This algorithm computes a Givens rotation,
with a corner case defined in a particularly important way.

\begin{enumerate}
\item Given a nonzero $2$-by-$1$ vector $x$, $\mathbf{givens}(m,i_{1},i_{2},x)$
constructs an $m$-by-$m$ Givens rotation whose submatrix lying in
rows and columns $i_{1}$ and $i_{2}$ is a 2-by-2 Givens rotation
$G=\left[\begin{smallmatrix} c & -s \\ s & c \end{smallmatrix}\right]$
with angle in $[0,\pi)$ such that $G^{T}x=(\pm\|x\|,0)^{T}$.
\item $\mathbf{givens}(m,i_{1},i_{2},(0,0)^{T})$ is defined to equal $\mathbf{givens}(m,i_{1},i_{2},(0,1)^{T})$---it
rotates through an angle of $\frac{\pi}{2}$.
\end{enumerate}
\end{alg}

The choice to rotate through $\frac{\pi}{2}$ when $x=(0,0)^{T}$
allows the following algorithm to handle Cases 1 and 2 of the preferred
QR factorization in a uniform way.

\begin{alg}[$\mathbf{qr\_step}$]\label{alg:implicitqrtridiagonal}Given
an $m$-by-$m$ real symmetric tridiagonal matrix $A$ and a shift
$\lambda\in\mathbb{R}$, the following algorithm performs one QR step.
See Theorem \ref{thm:implicitqrworks}. It is an extension of the
idea on pp. 418--420 of \cite{MR1417720}.

\lstset{columns=flexible,escapechar=\%,mathescape,numberblanklines=false,numbers=left,numberstyle=\tiny}
\begin{lstlisting}
%$\bar{A} := A-\lambda I$;%
%for $i=1,\dots,m-1$%
    %$v=\begin{cases} \bar{A}(i:i+1,i) & \text{if $i=1$ or $\bar{A}(i:i+1,i-1)$ is the zero vector} \\ \bar{A}(i:i+1,i-1) & \text{otherwise;} \end{cases}$%
%\label{line:computeG}%    %$G_i := \mathbf{givens}(m,i,i+1,v);$%
    %$\bar{A}:=G_i^T\bar{A}G_i$;%
%end for%
%$\bar{A}:=\bar{A}+\lambda I$;%
\end{lstlisting}

\end{alg}

\begin{thm}
\label{thm:implicitqrworks}Run Algorithm \ref{alg:implicitqrtridiagonal}
to compute $G_{1},\dots,G_{m-1}$ and $\bar{A}$, and define $Q=G_{1}\cdots G_{m-1}$
and $R=Q^{T}(A-\lambda I)$. Then $A-\lambda I=QR$ is a preferred
QR factorization and (\ref{eq:qrtridiagonal1}) and (\ref{eq:qrtridiagonal2})
are satisfied.
\end{thm}
\begin{proof}
The proof is broken into two cases.
\begin{description}
\item [{Case~1:~$\lambda$~is~not~an~eigenvalue~of~$A$.}] The
proof is by induction on $r$, the number of unreduced blocks of $A$.

The base case is $r=1$. In this case, the algorithm executes the
usual bulge-chasing QR step. See \cite{MR1417720}.

For the inductive step, suppose $r>1$ and assume that the theorem
has been proved for smaller values of $r$. Let $s$ be the number
of rows and columns in the final block $A_{r}$. The first $m-s-1$
executions of the loop neither observe nor modify the final $s$ rows
or columns of $\bar{A}$ and by induction compute $Q_{1},\dots,Q_{r-1}$
and $R_{1},\dots,R_{r-1}$ of the preferred QR factorization. At the
beginning of the loop with $i=m-s$, the $(m-s,m-s-1)$-entry of $\bar{A}$
is nonzero because $A-\lambda I$ is invertible, and the $(m-s+1,m-s-1)$-entry
of $\bar{A}$ is zero, so $G_{m-r}$ is set to an identity matrix.
When the loop continues with $i=m-s+1$, the algorithm proceeds as
in \cite{MR1417720}.

\item [{Case~2:~$\lambda$~is~an~eigenvalue~of~$A$.}] Let $l$
be the index of the first column of $A$ that is a linear combination
of the earlier columns. The algorithm proceeds as in Case 1 until
the loop with $i=l$. At that point, the leading $(l-1)$-by-$(l-1)$
principal submatrix of the new tridiagonal matrix is fixed (ignoring
the addition of $\lambda I$ at the very end), and considering the
proof of Theorem \ref{thm:qrstepperfectshift}, the $l$th row of
$\bar{A}$ is all zeros. Therefore, $G_{l}$ is a rotation by $\frac{\pi}{2}$,
which has the effect of swapping rows $l$ and $l+1$ of $\bar{A}$.
Now, the $(l+1)$st row of $\bar{A}$ is all zeros, so by induction,
all remaining Givens rotations have angle $\frac{\pi}{2}$ and the
row of zeros is pushed to the bottom of $\bar{A}$. This constructs
$Q$ and $R$ as in Case 2 of Definition \ref{def:preferredqr}. $\qed$
\end{description}
\end{proof}

\subsection{\label{sub:svdstep}SVD step}

The SVD algorithm of Golub and Kahan starts with the idea of applying
the QR step to $B^{T}B-\sigma^{2}I$. Because the formation of $B^{T}B$
is problematic in floating-point, the QR step must be executed implicitly,
working directly with the entries of $B$. The following definition
of the SVD step is unconventional but equivalent to the usual definition.

\begin{defn}
Let $B$ be a real bidiagonal matrix (either upper bidiagonal or lower
bidiagonal) and $\sigma$ a nonnegative real number. Matrices $\bar{B}$,
$S$, and $T$ are obtained from an \emph{SVD step} if \[
BB^{T}-\sigma^{2}I=SL^{T}\quad\text{and}\quad B^{T}B-\sigma^{2}I=TR\]
are preferred QR factorizations and\[
\bar{B}=S^{T}BT.\]

\end{defn}
Note that\[
\bar{B}\bar{B}^{T}=L^{T}S+\sigma^{2}I\quad\text{and}\quad\bar{B}^{T}\bar{B}=RT+\sigma^{2}I,\]
i.e., an SVD step on $B$ implicitly computes QR steps for $BB^{T}$
and $B^{T}B$.

\begin{thm}
The SVD step exists and is uniquely defined.
\end{thm}
\begin{proof}
This follows immediately from the existence and uniqueness of the
preferred QR factorization. $\qed$
\end{proof}
{}

\begin{thm}
\label{thm:bbarisbidiagonal}If $B$ is upper bidiagonal, then $\bar{B}$
is upper bidiagonal. If $B$ is lower bidiagonal, then $\bar{B}$
is lower bidiagonal.
\end{thm}
The proof is presented after Lemma \ref{lem:svdstepequivalenttoqrsteponbbt}
below.

Before going any further with the SVD step, we need a utility routine.

\begin{alg}[$\mathbf{bulge\_start}$]\label{alg:bulgeinducingvector}Given
a $2$-by-$1$ vector $x$ and a shift $\sigma\geq0$, $\mathbf{bulge\_start}$
computes a vector colinear with $(x_{1}^{2}-\sigma^{2},x_{1}x_{2})^{T}$.
If $(x_{1}^{2}-\sigma^{2},x_{1}x_{2})^{T}$ is the zero vector, then
$\mathbf{bulge\_start}$ returns the zero vector.\end{alg}

The implementation of $\mathbf{bulge\_start}$ is omitted. LAPACK's
DBDSQR provides guidance \cite{323215}.

The following algorithm computes an SVD step for an upper bidiagonal
matrix. It can also handle lower bidiagonal matrices by taking transposes
as appropriate.

\begin{alg}[$\mathbf{svd\_step}$]\label{alg:svdstep}Given an $m$-by-$m$
upper bidiagonal matrix $B$ and a shift $\sigma\geq0$, the following
algorithm computes one SVD step. See Theorem \ref{thm:svdstepworks}.

\lstset{columns=flexible,escapechar=\%,mathescape,numberblanklines=false,numbers=left,numberstyle=\tiny}
\begin{lstlisting}
%$\bar{B}:=B$;%

%for $i:=1,\dots,m-1$%

    %$v := \begin{cases} \mathbf{bulge\_start}(\bar{B}(i,i:i+1)^T,\sigma) & \mbox{if $i=1$ or $\bar{B}(i-1,i:i+1)=(0,0)$} \\ \bar{B}(i-1,i:i+1)^T & \text{otherwise;} \end{cases}$%
%\label{line:computeT}%    %$T_i:=\mathbf{givens}(m,i,i+1,v)$;%
    %\label{line:rotateBonright}%%$\bar{B}:=\bar{B}T_i$;%

    %$u := \begin{cases} \mathbf{bulge\_start}(\bar{B}(i:i+1,i+1),\sigma) & \text{if $\bar{B}(i:i+1,i)=(0,0)^T$} \\ \bar{B}(i:i+1,i) & \text{otherwise;}\end{cases}$%
%\label{line:computeS}%    %$S_i:=\mathbf{givens}(m,i,i+1,u)$;%
    %\label{line:rotateBonleft}%%$\bar{B}:=S_i^T\bar{B}$;%

%end for%

%$S:=S_1\cdots S_{m-1};\;\;T:=T_1\cdots T_{m-1}$;%
\end{lstlisting}

\end{alg}

\begin{thm}
\label{thm:svdstepworks}Given a real upper bidiagonal matrix $B$
and a shift $\sigma\geq0$,
\begin{enumerate}
\item Algorithm \ref{alg:svdstep} computes an SVD step.
\item $\bar{B}$ is real upper bidiagonal.
\end{enumerate}
\end{thm}
The proof follows immediately from Theorem \ref{thm:implicitqrworks}
and the following three lemmas.

\begin{lem}
\label{lem:svdstepproducesbidiagonal}Upon completion of line \ref{line:rotateBonright},
$\bar{B}$ is upper bidiagonal, with the possible exception of a {}``bulge''
at $(i+1,i)$. Upon completion of line \ref{line:rotateBonleft},
$\bar{B}$ is upper bidiagonal, with the possible exception of a {}``bulge''
at $(i,i+2)$ if $i<m-1$. Upon completion of the entire algorithm,
$\bar{B}$ is upper bidiagonal.
\end{lem}
The proof is omitted; the bulge-chasing nature of the algorithm is
standard.

\begin{lem}
\label{lem:svdstepequivalenttoqrsteponbtb}Let $B$ be an $m$-by-$m$
real upper bidiagonal matrix and $\sigma$ a nonnegative real number.
Run Algorithm \ref{alg:svdstep} to produce $T_{1},\dots,T_{m-1}$,
and run Algorithm \ref{alg:implicitqrtridiagonal} with $A=B^{T}B$
and $\lambda=\sigma^{2}$ to produce $G_{1},\dots,G_{m-1}$. Then
$T_{i}=G_{i}$ for $i=1,\dots,m-1$.
\end{lem}
The proof can be found in Appendix \ref{sec:proofs}.

{}

\begin{lem}
\label{lem:svdstepequivalenttoqrsteponbbt}Let $B$ be an $m$-by-$m$
real upper bidiagonal matrix and $\sigma$ a nonnegative real number.
Run Algorithm \ref{alg:svdstep} to produce $S_{1},\dots,S_{m-1}$,
and run Algorithm \ref{alg:implicitqrtridiagonal} with $A=BB^{T}$
and $\lambda=\sigma^{2}$ to produce $G_{1},\dots,G_{m-1}$. Then
$S_{i}=G_{i}$ for $i=1,\dots,m-1$.
\end{lem}
The proof can be found in Appendix \ref{sec:proofs}.

\medskip

\noindent\textit{Proof of Theorem \ref{thm:bbarisbidiagonal}} If
$B$ is upper bidiagonal, then $\bar{B}$ can be obtained from Algorithm
\ref{alg:svdstep}, which produces upper bidiagonal matrices. If $B$
is lower bidiagonal, then apply the same argument to $B^{T}$. $\qed$

\section{\label{sec:qrstep}Phase II: Algorithm $\mathbf{csd\_step}$}

Now we can return to the CSD algorithm. Phase I, which was already
seen, transforms the original partitioned unitary matrix to bidiagonal
block form. Phase II, which is developed now, iteratively applies
the SVD step to each of the four blocks of this matrix. Algorithm
$\mathbf{csd\_step}$ executes a single step in the iteration.

\subsection{Existence of the CSD step}

\begin{defn}
Let $\left[\begin{array}{cc}
B_{11} & B_{12}^{}\\
B_{21} & B_{22}\end{array}\right]$ be a real orthogonal matrix in bidiagonal block form and $\mu$ and
$\nu$ nonnegative shifts satisfying $\mu^{2}+\nu^{2}=1$. The matrix
equation\[
\left[\begin{array}{cc}
\bar{B}_{11} & \bar{B}_{12}\\
\bar{B}_{21} & \bar{B}_{22}\end{array}\right]=\left[\begin{array}{cc}
D_{1}\\
 & D_{2}\end{array}\right]\left[\begin{array}{cc}
S_{1}\\
 & S_{2}\end{array}\right]^{T}\left[\begin{array}{cc}
B_{11} & B_{12}\\
B_{21} & B_{22}\end{array}\right]\left[\begin{array}{cc}
T_{1}\\
 & T_{2}\end{array}\right]\left[\begin{array}{cc}
E_{1}\\
 & E_{2}\end{array}\right]\]
effects a \emph{CSD step} if
\begin{enumerate}
\item $B_{11}\mapsto S_{1}^{T}B_{11}T_{1}$ and $B_{22}\mapsto S_{2}^{T}B_{22}T_{2}$
are SVD steps with shift $\mu$,
\item $B_{12}\mapsto S_{1}^{T}B_{12}T_{2}$ and $B_{21}\mapsto S_{2}^{T}B_{21}T_{1}$
are SVD steps with shift $\nu$,
\item $D_{1}$, $D_{2}$, $E_{1}$, and $E_{2}$ are diagonal signature
matrices, and
\item $\left[\begin{array}{cc}
\bar{B}_{11} & \bar{B}_{12}\\
\bar{B}_{21} & \bar{B}_{22}\end{array}\right]$ is a matrix in bidiagonal block form.
\end{enumerate}
\end{defn}
The existence of the CSD step is not obvious at first glance. It depends
on several SVD steps being related in very specific ways, e.g., $B_{11}$
and $B_{21}$ having the same right orthogonal factor $T_{1}$. The
following theorem establishes that yes, indeed, the CSD step exists,
and its proof makes clear the necessity of the restriction $\mu^{2}+\nu^{2}=1$.

\begin{thm}
\label{thm:existenceanduniquenessofcsdstep}The CSD step exists and
the result\[
\left[\begin{array}{cc}
\bar{B}_{11} & \bar{B}_{12}\\
\bar{B}_{21} & \bar{B}_{22}\end{array}\right]\]
is uniquely defined.
\end{thm}
\begin{proof}
Begin with the identities\begin{align}
B_{11}^{T}B_{11}-\mu^{2}I & =-(B_{21}^{T}B_{21}-\nu^{2}I),\label{eq:blockidentities1}\\
B_{22}^{T}B_{22}-\mu^{2}I & =-(B_{12}^{T}B_{12}-\nu^{2}I),\\
B_{11}B_{11}^{T}-\mu^{2}I & =-(B_{12}B_{12}^{T}-\nu^{2}I),\\
B_{22}B_{22}^{T}-\mu^{2}I & =-(B_{21}B_{21}^{T}-\nu^{2}I).\label{eq:blockidentities4}\end{align}
Each is proved using orthogonality and the relation $\mu^{2}+\nu^{2}=1$.
For example, the orthogonality of $\left[\begin{smallmatrix} B_{11} & B_{12} \\ B_{21} & B_{22} \end{smallmatrix}\right]$
implies $B_{11}^{T}B_{11}+B_{21}^{T}B_{21}=I$, and splitting the
right-hand-side $I$ into $\mu^{2}I+\nu^{2}I$ and rearranging gives
(\ref{eq:blockidentities1}).

Define $\bar{B}_{ij}=S_{ij}B_{ij}T_{ij}$ by an SVD step with the
appropriate shift ($\mu$ if $i=j$ or $\nu$ if $i\neq j$). This
produces a total of eight orthogonal factors $S_{ij}$, $T_{ij}$,
but only four are required for the CSD step. In fact, by uniqueness
of the preferred QR factorization, $T_{11}=T_{21}$, $T_{12}=T_{22}$,
$S_{11}=S_{12}$, and $S_{21}=S_{22}$. (For example, $T_{11}$ is
the orthogonal factor in the preferred QR factorization $B_{11}^{T}B_{11}-\mu^{2}I=T_{11}R_{11}$,
and $T_{21}$ is the orthogonal factor in the preferred QR factorization
$B_{21}^{T}B_{21}-\nu^{2}I=T_{21}R_{21}$. Considering (\ref{eq:blockidentities1})
and the uniqueness of the preferred QR factorization, we must have
$T_{21}=T_{11}$ and $R_{21}=-R_{11}$.) Hence, it is legal to define
$T_{1}=T_{11}=T_{21}$, $T_{2}=T_{12}=T_{22}$, $S_{1}=S_{11}=S_{12}$,
and $S_{2}=S_{21}=S_{22}$.

Finally, $D_{1}$, $D_{2}$, $E_{1}$, and $E_{2}$ are designed to
fix the sign pattern required for a matrix in bidiagonal block form.
Their existence is guaranteed by Theorem \ref{thm:signsarenothingtofear}. 

Regarding uniqueness, $S_{1}$, $S_{2}$, $T_{1}$, and $T_{2}$ are
uniquely defined by the preferred QR factorization, so $S_{i}^{T}B_{ij}T_{j}$,
$i,j=1,2$, are uniquely defined. This uniquely defines the absolute
values of the entries of $\bar{B}_{ij}$, $i,j=1,2$, and the signs
are specified by the definition of bidiagonal block form. $\qed$
\end{proof}
The obvious way to compute a CSD step is to compute four SVD steps
and then to combine them together as in the proof. Of course, when
working in floating-point, the identities such as $T_{11}=T_{21}$
typically will not hold exactly. In fact, when singular values are
clustered, the computed $T_{11}$ and $T_{21}$ may not even bear
a close resemblance.

Our approach is to interleave the computation of the four SVD steps,
taking care to compute $S_{1}$, $S_{2}$, $T_{1}$, and $T_{2}$
once and only once. We find that when one block of a matrix provides
unreliable information (specifically, a very short vector whose direction
is required for a Givens rotation), another block may come to the
rescue, providing more reliable information. Hence, the redundancy
in the partitioned orthogonal matrix, rather than being a stumbling
block, is actually an aid to stability.

\subsection{Algorithm specification}

The following is a specification for Algorithm $\mathbf{csd\_step}$,
which accomplishes one step in Phase II of the CSD algorithm.

\begin{specification}\label{spec:csd_step}The input to $\mathbf{csd\_step}$
should consist of 

\begin{enumerate}
\item $\theta^{(n)}\in[0,\frac{\pi}{2}]^{q}$ and $\phi^{(n)}\in[0,\frac{\pi}{2}]^{q-1}$,
implicitly defining a matrix in bidiagonal block form $\left[\begin{smallmatrix} B^{(n)}_{11} & B^{(n)}_{12} \\ B^{(n)}_{21} & B^{(n)}_{22} \end{smallmatrix}\right]$,
\item integers $1\leq\underline{i}<\overline{i}\leq q$ identifying the
current block in a partially deflated matrix---$\left[\begin{smallmatrix}B^{(n)}_{11}&B^{(n)}_{12}\\B^{(n)}_{21}&B^{(n)}_{22}\end{smallmatrix}\right]$
must have the form\begin{equation}
\left[\begin{array}{clc|crc}
\mbox{\Large{$\ast$}} & 0 &  & \mbox{\Large{$\ast$}}\\
 & B_{11}^{(n)}(\underline{i}:\overline{i},\underline{i}:\overline{i}) & 0 & 0 & B_{12}^{(n)}(\underline{i}:\overline{i},\underline{i}:\overline{i})\\
 &  & \mbox{\Large{$\ast$}} &  & 0 & \mbox{\Large{$\ast$}}\\
\hline \mbox{\Large{$\ast$}} & 0 &  & \mbox{\Large{$\ast$}}\\
 & B_{21}^{(n)}(\underline{i}:\overline{i},\underline{i}:\overline{i}) & 0 & 0 & B_{22}^{(n)}(\underline{i}:\overline{i},\underline{i}:\overline{i})\\
 &  & \mbox{\Large{$\ast$}} &  & 0 & \mbox{\Large{$\ast$}}\end{array}\right],\label{eq:splitting}\end{equation}

\item shifts $0\leq\mu,\nu\leq1$ satisfying $\mu^{2}+\nu^{2}=1$.
\end{enumerate}
The algorithm should compute one CSD step with shifts $\mu$ and $\nu$
to effect \begin{equation}
\left[\begin{array}{l|r}
B_{11}^{(n)}(\underline{i}:\overline{i},\underline{i}:\overline{i}) & B_{12}^{(n)}(\underline{i}:\overline{i},\underline{i}:\overline{i})\\
\hline B_{21}^{(n)}(\underline{i}:\overline{i},\underline{i}:\overline{i}) & B_{22}^{(n)}(\underline{i}:\overline{i},\underline{i}:\overline{i})\end{array}\right]\mapsto\left[\begin{array}{l|r}
\bar{B}_{11} & \bar{B}_{12}\\
\hline \bar{B}_{21} & \bar{B}_{22}\end{array}\right].\label{eq:aftercsdstep}\end{equation}
The output should consist of $\theta^{(n+1)}\in[0,\frac{\pi}{2}]^{q}$
and $\phi^{(n+1)}\in[0,\frac{\pi}{2}]^{q-1}$, implicitly defining
the matrix $\left[\begin{smallmatrix} B^{(n+1)}_{11} & B^{(n+1)}_{12} \\ B^{(n+1)}_{21} & B^{(n+1)}_{22} \end{smallmatrix}\right]$
that results from replacing the appropriate submatrices from (\ref{eq:splitting})
with the corresponding submatrices of (\ref{eq:aftercsdstep}), and
orthogonal matrices $U_{1}^{(n+1)},U_{2}^{(n+1)},V_{1}^{(n+1)},V_{2}^{(n+1)}$
such that\begin{multline}
\left[\begin{array}{cc}
B_{11}^{(n+1)} & B_{12}^{(n+1)}\\
B_{21}^{(n+1)} & B_{22}^{(n+1)}\end{array}\right]=\left[\begin{array}{cc}
U_{1}^{(n+1)}\\
 & U_{2}^{(n+1)}\end{array}\right]^{*}\left[\begin{array}{cc}
B_{11}^{(n)} & B_{12}^{(n)}\\
B_{21}^{(n)} & B_{22}^{(n)}\end{array}\right]\left[\begin{array}{cc}
V_{1}^{(n+1)}\\
 & V_{2}^{(n+1)}\end{array}\right].\end{multline}
 \end{specification}

\subsection{The algorithm}

The naive idea for implementing the above specification is to compute
four SVD steps separately. As mentioned in the proof of Theorem \ref{thm:existenceanduniquenessofcsdstep},
this produces eight orthogonal factors when only four are needed.
In fact, in light of (\ref{eq:blockidentities1})--(\ref{eq:blockidentities4})
and Theorem \ref{thm:productofgivens}, the Givens rotations defining
$S_{1},S_{2},T_{1}$, and $T_{2}$ come in identical pairs. Of course,
in floating-point, the paired Givens rotations would not actually
be identical. Hence, the naive algorithm suffers from the standpoints
of efficiency (each Givens rotation is needlessly computed twice)
and stability (what happens when two computed Givens rotations disagree?).

Our solution, suggested by Fig.~\ref{fig:implicitqrpattern}, is
to execute the four SVD steps simultaneously, computing each Givens
rotation once and only once through a bulge-chasing procedure that
treats all four blocks with equal regard.

The algorithm makes use of a routine called $\mathbf{merge}$. In
the absence of roundoff error, $\mathbf{merge}$ is essentially a
no-op; given two vectors in any one-dimensional subspace, it returns
a vector in the same subspace. (And $\mathbf{merge}((0,0)^{T},(0,0)^{T})=(0,0)^{T}$.)
In the presence of roundoff error, $\mathbf{merge}$ is used to ameliorate
small differences resulting from previous roundoff errors.

\begin{alg}[$\mathbf{csd\_step}$]\label{alg:csdstep}This algorithm performs one CSD step on a matrix in bidiagonal block form and satisfies Specification \ref{spec:csd_step}. Its correctness is established in Theorem \ref{thm:csdstepsatisfiesspec}.

\medskip{}

\lstset{columns=flexible,escapechar=\%,mathescape,numberblanklines=false,numbers=left,numberstyle=\tiny}
\begin{lstlisting}
%\label{line:qrstepbdbfromangles}%%Explicitly construct $\left[\begin{smallmatrix}B_{11}&B_{12}\\B_{21}&B_{22}\end{smallmatrix}\right]$ from $(\theta^{(n)},\phi^{(n)})$;%

%\label{line:qrstepgivensbegin}%%Initialize $\left[\begin{smallmatrix}\bar{B}_{11}&\bar{B}_{12}\\\bar{B}_{21}&\bar{B}_{22}\end{smallmatrix}\right]:=\left[\begin{smallmatrix}B_{11}&B_{12}\\B_{21}&B_{22}\end{smallmatrix}\right]$;%

%\label{line:definev11a}%%$v_{11} := \mathbf{bulge\_start}(\bar{B}_{11}(\underline{i},\underline{i}:\underline{i}+1)^T,\mu)$;%
%\label{line:definev21a}%%$v_{21} := \mathbf{bulge\_start}(\bar{B}_{21}(\underline{i},\underline{i}:\underline{i}+1)^T,\nu)$;%
%\label{line:merge1}%%$v_1 := \mathbf{merge}(v_{11},v_{21})$;%
%$T_{1,\underline{i}} := \mathbf{givens}\left(q, \underline{i}, \underline{i}+1, v_1\right)$;%
%$\left[\begin{smallmatrix}\bar{B}_{11}&\bar{B}_{12}\\\bar{B}_{21}&\bar{B}_{22}\end{smallmatrix}\right]:=\left[\begin{smallmatrix}\bar{B}_{11}&\bar{B}_{12}\\\bar{B}_{21}&\bar{B}_{22}\end{smallmatrix}\right]\left[\begin{smallmatrix}T_{1,\underline{i}}&\\&I\end{smallmatrix}\right]$;%
%\label{line:defineu11a}%%$u_{11} := \begin{cases} \bar{B}_{11}(\underline{i}:\underline{i}+1,\underline{i}) & \text{if nonzero} \\ \mathbf{bulge\_start}(\bar{B}_{11}(\underline{i}:\underline{i}+1,\underline{i}+1),\mu) & \text{otherwise}; \end{cases}$%
%\label{line:defineu12a}%%$u_{12} := \mathbf{bulge\_start}(\bar{B}_{12}(\underline{i}:\underline{i}+1,\underline{i}),\nu)$;%
%\label{line:defineu21a}%%$u_{21} := \begin{cases} \bar{B}_{21}(\underline{i}:\underline{i}+1,\underline{i}) & \text{if nonzero} \\ \mathbf{bulge\_start}(\bar{B}_{21}(\underline{i}:\underline{i}+1,\underline{i}+1),\nu) & \text{otherwise}; \end{cases}$%
%\label{line:defineu22a}%%$u_{22} := \mathbf{bulge\_start}(\bar{B}_{22}(\underline{i}:\underline{i}+1,\underline{i}),\mu)$;%
%\label{line:merge2}%%$u_1 := \mathbf{merge}(u_{11},u_{12})$; $u_2 := \mathbf{merge}(u_{21},u_{22})$;%
%$S_{1,\underline{i}}:=\mathbf{givens}(q,\underline{i},\underline{i}+1,u_1)$; $S_{2,\underline{i}}:=\mathbf{givens}(q,\underline{i},\underline{i}+1,u_2)$;%
%$\left[\begin{smallmatrix}\bar{B}_{11}&\bar{B}_{12}\\\bar{B}_{21}&\bar{B}_{22}\end{smallmatrix}\right] := \left[\begin{smallmatrix}S_{1,\underline{i}}&\\&S_{2,\underline{i}}\end{smallmatrix}\right]^T \left[\begin{smallmatrix}\bar{B}_{11}&\bar{B}_{12}\\\bar{B}_{21}&\bar{B}_{22}\end{smallmatrix}\right]$;%

%for $i:=\underline{i}+1,\dots,\overline{i}-1$%

%\label{line:definev11b}%    %$v_{11} := \begin{cases} \bar{B}_{11}(i-1,i:i+1) & \text{if nonzero} \\ \mathbf{bulge\_start}(\bar{B}_{11}(i,i:i+1),\mu) & \text{otherwise;} \end{cases}$%
%\label{line:definev21b}%    %$v_{21} := \begin{cases} \bar{B}_{21}(i-1,i:i+1) & \text{if nonzero} \\ \mathbf{bulge\_start}(\bar{B}_{21}(i,i:i+1),\nu) & \text{otherwise;} \end{cases}$%
%\label{line:definev12b}%    %$v_{12} := \begin{cases} \bar{B}_{12}(i-1,i-1:i) & \text{if nonzero} \\ \mathbf{bulge\_start}(\bar{B}_{12}(i,i-1:i),\nu) & \text{otherwise;} \end{cases}$%
%\label{line:definev22b}%    %$v_{22} := \begin{cases} \bar{B}_{22}(i-1,i-1:i) & \text{if nonzero} \\ \mathbf{bulge\_start}(\bar{B}_{22}(i,i-1:i),\mu) & \text{otherwise;} \end{cases}$%
%\label{line:merge3}%    %$v_1 := \mathbf{merge}(v_{11},v_{21})$; $v_2 := \mathbf{merge}(v_{12},v_{22})$;%
    %$T_{1,i} := \mathbf{givens}(q,i,i+1,v_1)$; $T_{2,i-1}:=\mathbf{givens}(q,i-1,i,v_2)$;%
    %$\left[\begin{smallmatrix}\bar{B}_{11}&\bar{B}_{12}\\\bar{B}_{21}&\bar{B}_{22}\end{smallmatrix}\right] :=  \left[\begin{smallmatrix}\bar{B}_{11}&\bar{B}_{12}\\\bar{B}_{21}&\bar{B}_{22}\end{smallmatrix}\right] \left[\begin{smallmatrix}T_{1,i}&\\&T_{2,i-1}\end{smallmatrix}\right]$;%

%\label{line:defineu11b}%    %$u_{11} := \begin{cases} \bar{B}_{11}(i:i+1,i) & \text{if nonzero} \\ \mathbf{bulge\_start}(\bar{B}_{11}(i:i+1,i+1),\mu) & \text{otherwise;} \end{cases}$%
%\label{line:defineu12b}%    %$u_{12} := \begin{cases} \bar{B}_{12}(i:i+1,i-1) & \text{if nonzero} \\ \mathbf{bulge\_start}(\bar{B}_{12}(i:i+1,i),\nu) & \text{otherwise;} \end{cases}$%
%\label{line:defineu21b}%    %$u_{21} := \begin{cases} \bar{B}_{21}(i:i+1,i) & \text{if nonzero} \\ \mathbf{bulge\_start}(\bar{B}_{21}(i:i+1,i+1),\nu) & \text{otherwise;} \end{cases}$%
%\label{line:defineu22b}%    %$u_{22} := \begin{cases} \bar{B}_{22}(i:i+1,i-1) & \text{if nonzero} \\ \mathbf{bulge\_start}(\bar{B}_{22}(i:i+1,i),\mu) & \text{otherwise;} \end{cases}$%
%\label{line:merge4}%    %$u_1 := \mathbf{merge}(u_{11},u_{12})$; $u_2 := \mathbf{merge}(u_{21},u_{22})$;%
    %$S_{1,i}:=\mathbf{givens}(q,i,i+1,u_1)$; $S_{2,i}:=\mathbf{givens}(q,i,i+1,u_2)$;%
    %$\left[\begin{smallmatrix}\bar{B}_{11}&\bar{B}_{12}\\\bar{B}_{21}&\bar{B}_{22}\end{smallmatrix}\right] := \left[\begin{smallmatrix}S_{1,i}&\\&S_{2,i}\end{smallmatrix}\right]^T \left[\begin{smallmatrix}\bar{B}_{11}&\bar{B}_{12}\\\bar{B}_{21}&\bar{B}_{22}\end{smallmatrix}\right]$;%

%\label{line:qrstepgivensend}%%end for%

%\label{line:definev12c}%%$v_{12} := \begin{cases} \bar{B}_{12}(\overline{i}-1,\overline{i}-1:\overline{i}) & \text{if nonzero} \\ \mathbf{bulge\_start}(\bar{B}_{12}(\overline{i},\overline{i}-1:\overline{i}),\nu) & \text{otherwise;} \end{cases}$%
%\label{line:definev22c}%%$v_{22} := \begin{cases} \bar{B}_{22}(\overline{i}-1,\overline{i}-1:\overline{i}) & \text{if nonzero} \\ \mathbf{bulge\_start}(\bar{B}_{22}(\overline{i},\overline{i}-1:\overline{i}),\mu) & \text{otherwise;} \end{cases}$%
%\label{line:merge5}%%$v_2 := \mathbf{merge}(v_{12},v_{22})$;%
%$T_{2,\overline{i}-1}:=\mathbf{givens}(q,\overline{i}-1,\overline{i},v_2)$;%
%$\left[\begin{smallmatrix}\bar{B}_{11}&\bar{B}_{12}\\\bar{B}_{21}&\bar{B}_{22}\end{smallmatrix}\right] :=  \left[\begin{smallmatrix}\bar{B}_{11}&\bar{B}_{12}\\\bar{B}_{21}&\bar{B}_{22}\end{smallmatrix}\right] \left[\begin{smallmatrix}I&\\&T_{2,\overline{i}-1}\end{smallmatrix}\right]$;%

%\label{line:qrstepu}%%$U_1^{(n+1)}:=S_{1,\underline{i}}\cdots S_{1,\overline{i}-1};\;\;U_2^{(n+1)}:=S_{2,\underline{i}}\cdots S_{2,\overline{i}-1}$;%
%$V_1^{(n+1)}:=T_{1,\underline{i}}\cdots T_{1,\overline{i}-1};\;\;V_2^{(n+1)}:=T_{2,\underline{i}}\cdots T_{2,\overline{i}-1}$;%

%\label{line:signs}%%Fix signs in $\left[\begin{smallmatrix}\bar{B}_{11}&\bar{B}_{12}\\\bar{B}_{21}&\bar{B}_{22}\end{smallmatrix}\right]$ to match (\ref{eq:jacobiform}), negating columns of $U_1^{(n+1)}$, $U_2^{(n+1)}$, $V_1^{(n+1)}$, and $V_2^{(n+1)}$ as necessary;%

%\label{line:angles}%%Compute $(\theta^{(n+1)},\phi^{(n+1)})$ to implicitly represent $\left[\begin{smallmatrix}\bar{B}_{11}&\bar{B}_{12}\\\bar{B}_{21}&\bar{B}_{22}\end{smallmatrix}\right]$;%
\end{lstlisting}
\end{alg}

Note that it is possible and preferable to enforce signs and compute
$\theta^{(n+1)}$ and $\phi^{(n+1)}$ as Givens rotations are applied,
instead of waiting until the end of the algorithm. 

\begin{thm}
\label{thm:csdstepsatisfiesspec}If arithmetic operations are performed
exactly, then Algorithm $\mathbf{csd\_step}$ satisfies Specification
\ref{spec:csd_step}.
\end{thm}
\begin{proof}
The proof is organized into three parts. First, it is shown that the
algorithm computes SVD steps for the top-left and bottom-right blocks.
Then, it is shown that the algorithm computes SVD steps for the top-right
and bottom-left blocks. Finally, utilizing the proof of Theorem \ref{thm:existenceanduniquenessofcsdstep},
it is shown that the algorithm as stated simultaneously computes SVD
steps of all four blocks and hence a CSD step for the entire matrix.

For the first part, temporarily delete lines \ref{line:definev21a},
\ref{line:defineu12a}, \ref{line:defineu21a}, \ref{line:definev21b},
\ref{line:definev12b}, \ref{line:defineu12b}, \ref{line:defineu21b},
\ref{line:definev12c}, \ref{line:signs}, and \ref{line:angles}
and make the following replacements:

\begin{tabular}{lll}
\hline 
line & original & replacement\tabularnewline
\hline
\ref{line:merge1} & $v_{1}:=\mathbf{merge}(v_{11},v_{21})$; & $v_{1}:=v_{11}$;\tabularnewline
\ref{line:merge2} & $u_{1}:=\mathbf{merge}(u_{11},u_{12})$; $u_{2}:=\mathbf{merge}(u_{21},u_{22})$; & $u_{1}:=u_{11}$; $u_{2}:=u_{22}$;\tabularnewline
\ref{line:merge3} & $v_{1}:=\mathbf{merge}(v_{11},v_{21})$; $v_{2}:=\mathbf{merge}(v_{12},v_{22})$; & $v_{1}:=v_{11}$; $v_{2}:=v_{22}$;\tabularnewline
\ref{line:merge4} & $u_{1}:=\mathbf{merge}(u_{11},u_{12})$; $u_{2}:=\mathbf{merge}(u_{21},u_{22})$; & $u_{1}:=u_{11}$; $u_{2}:=u_{22}$;\tabularnewline
\ref{line:merge5} & $v_{2}:=\mathbf{merge}(v_{12},v_{22})$; & $v_{2}:=v_{22}$;\tabularnewline
\hline
\end{tabular}

\noindent{}All references to the top-right and bottom-left blocks
have been removed, and the resulting algorithm is equivalent to running
Algorithm $\mathbf{svd\_step}$ on the top-left and bottom-right blocks.
Hence, $U_{1}^{(n+1)}$ and $V_{1}^{(n+1)}$ are the outer factors
from an SVD step on the top-left block, and $U_{2}^{(n+1)}$ and $V_{2}^{(n+1)}$
are the outer factors from an SVD step on the bottom-right block.

For the second part, revert back to the original algorithm but make
the following changes, intended to focus attention on the top-right
and bottom-left blocks: Delete lines \ref{line:definev11a}, \ref{line:defineu11a},
\ref{line:defineu22a}, \ref{line:definev11b}, \ref{line:definev22b},
\ref{line:defineu11b}, \ref{line:defineu22b}, \ref{line:definev22c},
\ref{line:signs}, and \ref{line:angles} and make the following replacements:

\begin{tabular}{lll}
\hline 
line & original & replacement\tabularnewline
\hline
\ref{line:merge1} & $v_{1}:=\mathbf{merge}(v_{11},v_{21})$; & $v_{1}:=v_{21}$;\tabularnewline
\ref{line:merge2} & $u_{1}:=\mathbf{merge}(u_{11},u_{12})$; $u_{2}:=\mathbf{merge}(u_{21},u_{22})$; & $u_{1}:=u_{12}$; $u_{2}:=u_{21}$;\tabularnewline
\ref{line:merge3} & $v_{1}:=\mathbf{merge}(v_{11},v_{21})$; $v_{2}:=\mathbf{merge}(v_{12},v_{22})$; & $v_{1}:=v_{21}$; $v_{2}:=v_{12}$;\tabularnewline
\ref{line:merge4} & $u_{1}:=\mathbf{merge}(u_{11},u_{12})$; $u_{2}:=\mathbf{merge}(u_{21},u_{22})$; & $u_{1}:=u_{12}$; $u_{2}:=u_{21}$;\tabularnewline
\ref{line:merge5} & $v_{2}:=\mathbf{merge}(v_{12},v_{22})$; & $v_{2}:=v_{12}$;\tabularnewline
\hline
\end{tabular}

\noindent{}This time, the algorithm is equivalent to running Algorithm
$\mathbf{svd\_step}$ on the top-right and bottom-left blocks, and
so $U_{1}^{(n+1)}$ and $V_{2}^{(n+1)}$ are the outer factors from
an SVD step on the top-right block and $U_{2}^{(n+1)}$ and $V_{1}^{(n+1)}$
are the outer factors from an SVD step on the bottom-left block.

By the existence of the CSD step, the two versions of the algorithm
discussed above produce identical $U_{1}^{(n+1)}$, $U_{2}^{(n+1)}$,
$V_{1}^{(n+1)}$, and $V_{2}^{(n+1)}$. Hence, by Theorem \ref{thm:productofgivens},
the two versions produce identical Givens rotations. Therefore, in
each invocation of $\mathbf{merge}$ in the final algorithm, one of
the following holds: both vectors are nonzero and colinear, or one
vector is the zero vector and the other vector points along the second
coordinate axis, or both vectors equal the zero vector. In any case,
the call to $\mathbf{merge}$ is well defined. Therefore, the final
algorithm simultaneously computes SVD steps for all four blocks.

The final two lines of code could be implemented (inefficiently) with
a single call to $\mathbf{bidiagonalize}$. It is straightforward
to check that this would only change the signs of entries and would
compute $\theta^{(n+1)}$ and $\phi^{(n+1)}$. $\qed$
\end{proof}

\section{\label{sec:csd}Algorithm $\mathbf{csd}$}

Algorithm $\mathbf{csd}$ is the driver algorithm, responsible for
initiating Phase I and Phase II.

Given an $m$-by-$m$ unitary matrix $X$ and integers $p$, $q$
with $0\leq q\leq p$ and $p+q\leq m$, $\mathbf{csd}(X,p,q)$ should
compute $\theta=(\theta_{1},\dots,\theta_{q})\in[0,\frac{\pi}{2}]^{q}$,
$U_{1}\in U(p)$, $U_{2}\in U(m-p)$, $V_{1}\in U(q)$, and $V_{2}\in U(m-q)$
satisfying\[
X\approx\left[\begin{array}{cc}
U_{1}\\
 & U_{2}\end{array}\right]\left[\begin{array}{r|ccc}
C & S & 0 & 0\\
0 & 0 & I_{p-q} & 0\\
\hline -S & C & 0 & 0\\
0 & 0 & 0 & I_{m-p-q}\end{array}\right]\left[\begin{array}{cc}
V_{1}\\
 & V_{2}\end{array}\right]^{*},\]
with $C=\diag(\cos(\theta_{1}),\dots,\cos(\theta_{q}))$ and $S=\diag(\sin(\theta_{1}),\dots,\sin(\theta_{q}))$.
It would be preferable to replace the approximate equality with an
error bound, but a formal stability analysis is left to the future.

$\mathbf{csd}$'s responsibility is to invoke $\mathbf{bidiagonalize}$
once and then $\mathbf{csd\_step}$ repeatedly. The following code
fleshes out this outline.

\begin{alg}[$\mathbf{csd}$]
Given $X$, $p$, and $q$, the following algorithm computes the complete CS decomposition in terms of $\theta$, $U_1$, $U_2$, $V_1$, and $V_2$.
\lstset{columns=flexible,mathescape,numberblanklines=false,numbers=left,numberstyle=\tiny,escapechar=\%}

\medskip

\begin{lstlisting}
%\label{line:csdcallsbidiagonalize}%Find $\theta^{(0)}$, $\phi^{(0)}$, $P_1$, $P_2$, $Q_1$, and $Q_2$ with $\mathbf{bidiagonalize}$;
%\label{line:csdround1}If any $\theta^{(0)}_i$ or $\phi^{(0)}_i$ is negligibly different from $0$ or $\frac{\pi}{2}$, then round;%

$n:=0$;

Set $\underline{i}$ and $\overline{i}$ as in %(\ref{eq:iminimax})%;

%while $\overline{i}>1$%

    %if some $\theta^{(n)}_{i}$, $\underline{i} \leq i \leq \overline{i}$, is negligibly far from $\frac{\pi}{2}$, then%
        %$\mu:=0$; $\nu:=1$;%
    %elseif some $\theta^{(n)}_{i}$, $\underline{i} \leq i \leq \overline{i}$, is negligibly far from $0$, then%
        %$\mu:=1$; $\nu:=0$;%
    %else%
%\label{line:setshifts}%        %Select shifts $\mu$ and $\nu$ satisfying $\mu^2+\nu^2=1$;%    %(* see discussion *)%
    %end if%

%\label{line:csdcallsqrstep}%    %Compute $\theta^{(n+1)},\phi^{(n+1)},U^{(n+1)}_1,U^{(n+1)}_2,V^{(n+1)}_1,V^{(n+1)}_2$ with $\mathbf{csd\_step}$;%

    $n:=n+1$;
    %\label{line:csdround2}If any $\theta^{(n)}_i$ or $\phi^{(n)}_i$ is negligibly different from $0$ or $\frac{\pi}{2}$, then round;%
    Set $\underline{i}$ and $\overline{i}$ as in %(\ref{eq:iminimax})%;

end while

$U_1 := P_1 ((U^{(1)}_1U^{(2)}_1\cdots U^{(n)}_1) \oplus I_{p-q})$; $U_2 := P_2 ((U^{(1)}_2U^{(2)}_2\cdots U^{(n)}_2) \oplus I_{m-p-q})$;
$V_1 := Q_1 (V^{(1)}_1V^{(2)}_1\cdots V^{(n)}_1)$; $V_2 := Q_2 ((V^{(1)}_2V^{(2)}_2\cdots V^{(n)}_2) \oplus I_{m-2q})$;
\end{lstlisting}
\end{alg}

As promised, the algorithm calls $\mathbf{bidiagonalize}$ once and
$\mathbf{csd\_step}$ repeatedly. The subtleties are the deflation
procedure and the choice of shifts.

\paragraph{Deflation}

If the matrix obtained after $n$ CSD steps,\[
\left[\begin{array}{cc}
B_{11}^{(n)} & B_{12}^{(n)}\\
B_{21}^{(n)} & B_{22}^{(n)}\end{array}\right]=\left[\begin{array}{cc}
B_{11}(\theta^{(n)},\phi^{(n)}) & B_{12}(\theta^{(n)},\phi^{(n)})\\
B_{21}(\theta^{(n)},\phi^{(n)}) & B_{22}(\theta^{(n)},\phi^{(n)})\end{array}\right],\]
has the form (\ref{eq:splitting}), then the next CSD step can focus
on just a principal submatrix. The indices $\underline{i}$ and $\overline{i}$
are determined by $\phi^{(n)}$ as follows. \begin{equation}
\left\{ \begin{aligned} & \phi_{\overline{i}}^{(n)}=\cdots=\phi_{q-1}^{(n)}=0\\
 & \text{none of }\phi_{\underline{i}}^{(n)},\dots,\phi_{\overline{i}-1}^{(n)}\text{ equal }0\\
 & \underline{i}\text{ is as small as possible}\end{aligned}
\right.\label{eq:iminimax}\end{equation}
Hence, deflation occurs when some $\phi_{i}^{(n)}$ transitions from
nonzero to zero.

Note that some $\theta_{i}^{(n)}$ may equal $0$ or $\frac{\pi}{2}$
or some $\phi_{i}^{(n)}$ may equal $\frac{\pi}{2}$, producing zero
entries in the matrix without satisfying (\ref{eq:iminimax}). Fortunately,
in this case one of the blocks has a zero on its diagonal, the appropriate
shift is set to zero, and deflation occurs in the one block just as
in the bidiagonal SVD algorithm of Demmel and Kahan \cite{MR1057146}.
(This is the reason for Case 2 of the preferred QR factorization.)
It is easy to check that then one of the $\phi_{i}^{(n+1)}$ becomes
zero. Hence, as soon as a zero appears anywhere in any of the four
bidiagonal bands, the entire matrix in bidiagonal block form deflates
in at most one more step (assuming exact arithmetic).

\paragraph{Choice of shift}

There are two natural possibilities for choosing shifts in line \ref{line:setshifts}. 

\begin{itemize}
\item One possibility is to let $\mu$ or $\nu$ be a Wilkinson shift (the
smaller singular value of the trailing 2-by-2 submatrix of one of
the blocks) from which the other shift follows by $\mu^{2}+\nu^{2}=1$.
This seems to give preference to one block over the other three, but
as mentioned above, as soon as one block attains a zero, the other
three blocks follow immediately.
\item Another possibility is to let $\mu$ and $\nu$ be singular values
of appropriate blocks, i.e., {}``perfect shifts.'' This keeps the
number of CSD steps as small as possible \cite[p.~417]{MR1417720}
at the cost of extra singular value computations.
\end{itemize}
Based on what is known about the SVD problem and limited testing with
the CSD algorithm, Wilkinson shifts appear to be the better choice,
but more real world experience is desirable.

\section{\label{sec:numericalconcerns}On numerical stability}

So far, the input matrix $X$ has been assumed exactly unitary, and
exact arithmetic has been assumed as well. What happens in a more
realistic environment? The algorithm is designed for numerical stability.
All four blocks are considered simultaneously and with equal regard.
Nearly every computation is based on information from two different
sources, from which the algorithm can choose the more reliable source.

Below, some numerical issues and strategies are addressed. Then results
from numerical experiments are presented. A BLAS/\-LAPACK-based implementation
is available from the author's web site for further testing.

\subsection{Numerical issues and strategies}

\paragraph{The better of two vectors in $\mathbf{bidiagonalize}$.}

As mentioned in the discussion after Algorithm \ref{alg:bidiagonalize},
most of the Householder reflectors used in the bidiagonalization procedure
are determined by pairs of colinear vectors. If the input matrix $X$
is not exactly unitary, then some of the pairs of vectors will not
be exactly colinear. Each of the computations for $u_{1}$, $u_{2}$,
$v_{1}$, and $v_{2}$ in Algorithm \ref{alg:bidiagonalize} performs
an averaging of two different vectors in which the vector of greater
norm is weighted more heavily. (Vectors of greater norm tend to provide
more reliable information about direction.)

\paragraph{Implementation of $\mathbf{bulge\_start}$.}

The implementation of $\mathbf{bulge\_start}$ (Algorithm \ref{alg:bulgeinducingvector})
requires care to minimize roundoff error. LAPACK's DBDSQR provides
guidance \cite{323215}.

\paragraph{Implementation of $\mathbf{merge}$.}

In $\mathbf{csd\_step}$, most Givens rotations can be constructed
from either (or both) of two colinear vectors. The function $\mathbf{merge}$
in the pseudocode is shorthand for a somewhat more complicated procedure
in our implementation:

\begin{itemize}
\item If the Givens rotation is chasing an existing bulge in one block and
introducing a new bulge in the other block, then base the Givens rotation
entirely on the existing bulge. (Maintaining bidiagonal block form
is crucial.)
\item If the Givens rotation is chasing existing bulges in both blocks,
then take a weighted average of the two vectors in the call to $\mathbf{merge}$,
waiting the longer vector more heavily. (Vectors of greater norm provide
more reliable information about direction.)
\item If the Givens rotation is introducing new bulges into both blocks,
then base the computation solely on the block associated with the
smaller shift, either $\mu$ or $\nu$. (In particular, when one shift
is zero, this strategy avoids roundoff error in $\mathbf{bulge\_start}$.)
\end{itemize}

\paragraph{Representation of $\theta^{(n)}$ and $\phi^{(n)}$.}

Angles of $0$ and $\frac{\pi}{2}$ play a special role in the deflation
procedure. Because common floating-point architectures represent angles
near 0 more precisely than angles near $\frac{\pi}{2}$, it may be
advisable to store any angle $\psi$ as a pair $(\psi,\frac{\pi}{2}-\psi)$.
This may provide a minor improvement but appears to be optional.

\paragraph{Rounding of $\theta^{(n)}$ and $\phi^{(n)}$ in $\mathbf{csd}$.}

To encourage fast termination, some angles in $\theta^{(n)}$ and
$\phi^{(n)}$ may need to be rounded when they are negligibly far
from 0 or $\frac{\pi}{2}$. The best test for negligibility will be
the subject of future work.

\paragraph{Disagreement of singular values when using perfect shifts.}

If in $\mathbf{csd}$, the shifts $\mu$ and $\nu$ are chosen to
be perfect shifts, i.e., singular values of appropriate blocks, then
the computed shifts may not satisfy $\mu^{2}+\nu^{2}=1$ exactly in
the presence of roundoff error. Empirically, satisfying $\mu^{2}+\nu^{2}=1$
appears to be crucial. Either $\mu$ or $\nu$ should be set to a
singular value no greater than $\frac{1}{\sqrt{2}}$ and then the
other shift computed from $\mu^{2}+\nu^{2}=1$.

\subsection{Results of numerical experiments}

The sharing of singular vectors in (\ref{eq:csdsvd}) is often seen
as a hindrance to numerical computation. But in our algorithm, the
redundancy appears to be a source of robustness---when one block provides
questionable information, a neighboring block may provide more reliable
information. Numerical tests support this claim.

Our criteria for a stable CSD computation,\[
X\approx\left[\begin{array}{cc}
U_{1}\\
 & U_{2}\end{array}\right]\left[\begin{array}{r|ccc}
C & S & 0 & 0\\
0 & 0 & I_{p-q} & 0\\
\hline -S & C & 0 & 0\\
0 & 0 & 0 & I_{m-p-q}\end{array}\right]\left[\begin{array}{cc}
V_{1}\\
 & V_{2}\end{array}\right]^{*},\]
are the following: If $X$ is nearly unitary, \begin{equation}
\|X^{*}X-I_{m}\|_{2}=\varepsilon,\label{eq:stabilityinputcriterion}\end{equation}
then we desire $\theta=(\theta_{1},\dots,\theta_{q})$ and $U_{1},U_{2},V_{1},V_{2}$
such that\[
C=\diag(\cos(\theta_{1}),\dots,\cos(\theta_{q}))\qquad S=\diag(\sin(\theta_{1}),\dots,\sin(\theta_{q}))\]
\[
\|U_{1}^{*}U_{1}-I_{p}\|_{2}\approx\varepsilon\qquad\|U_{2}^{*}U_{2}-I_{m-p}\|_{2}\approx\varepsilon\]
\[
\|V_{1}^{*}V_{1}-I_{q}\|_{2}\approx\varepsilon\qquad\|V_{2}^{*}V_{2}-I_{m-q}\|_{2}\approx\varepsilon\]
\[
\left[\begin{array}{c}
C\\
0\end{array}\right]=U_{1}^{T}(X_{11}+E_{11})V_{1},\quad\|E_{11}\|_{2}\approx\varepsilon\]
\[
\left[\begin{array}{ccc}
S & 0 & 0\\
0 & I_{p-q} & 0\end{array}\right]=U_{1}^{T}(X_{12}+E_{12})V_{2},\quad\|E_{12}\|_{2}\approx\varepsilon\]
\[
\left[\begin{array}{c}
-S\\
0\end{array}\right]=U_{2}^{T}(X_{21}+E_{21})V_{1},\quad\|E_{21}\|_{2}\approx\varepsilon\]
\[
\left[\begin{array}{ccc}
C & 0 & 0\\
0 & 0 & I_{m-p-q}\end{array}\right]=U_{2}^{T}(X_{22}+E_{22})V_{2},\quad\|E_{22}\|_{2}\approx\varepsilon.\]

\paragraph{Van Loan's example.}

Our first test case is based on an example of Van Loan \cite{MR796639}.
Let \small\[
X_{11}
=
\left[
\begin{array}{D{.}{.}{12} D{.}{.}{12} D{.}{.}{13} D{.}{.}{12}}
\ \ 0.220508860423 &-0.114095899416 & 0.001410518052 &\ 0.309131888087\\
 0.075149984350 & 0.552192330457 & 0.309420137864 & 0.519525649668\\
 0.346099513974 &-0.465523358094 &-0.147474170901 & 0.284504924779\\
 0.200314808251 & 0.015869922033 & 0.063768831702 & 0.364621650530
\end{array}
\right]
\]
\[
X_{12}
=
\left[
\begin{array}{D{.}{.}{12} D{.}{.}{12} D{.}{.}{12} D{.}{.}{12}}
 0.123868614848 &-0.424487382687 & 0.756283107266 &-0.274401793502\\
 0.505660921957 & 0.028765021298 &-0.138696588123 & 0.219160328651\\
-0.068044487719 &-0.292950312278 &-0.202722377746 & 0.655183291894\\
-0.339461927716 &-0.307319405113 &-0.530848659627 &-0.575436177767
\end{array}
\right]
\]
\[
X_{21}
=
\left[
\begin{array}{D{.}{.}{12} D{.}{.}{12} D{.}{.}{12} D{.}{.}{12}}
-0.149903307775 & 0.456869095895 &-0.814555019070 & 0.205461483909\\
-0.132593956233 & 0.403919514293 & 0.374067025998 &-0.294979263882\\
 0.631588073183 & 0.226164206817 & 0.132173742848 & 0.047014825861\\
-0.588949720476 &-0.205112923304 & 0.239887841318 & 0.537774110108
\end{array}
\right]
\]
\[
X_{22}
=
\left[
\begin{array}{D{.}{.}{12} D{.}{.}{13} D{.}{.}{13} D{.}{.}{12}}
-0.211288905103 &-0.065095708488 &\ 0.064582503584 &\ 0.100169729053\\
-0.422173038686 &-0.565182436669 & 0.079260723473 & 0.297296887111\\
-0.473064671229 & 0.502284642254 & 0.218767959397 & 0.079539299401\\
-0.403033356444 & 0.250518329548 & 0.166101999167 & 0.107399029584
\end{array}
\right],
\]\normalsize{} and let $X=\left[\begin{smallmatrix}X_{11} & X_{12} \\ X_{21} & X_{22} \end{smallmatrix}\right]$.
Van Loan considered the submatrix $\left[\begin{smallmatrix}X_{11} \\ X_{21} \end{smallmatrix}\right]$. 

$X$ satisfies (\ref{eq:stabilityinputcriterion}) with $\varepsilon\approx3.4\times10^{-12}$.
Our implementation produces $\theta$, $U_{1}$, $U_{2}$, $V_{1}$,
and $V_{2}$ for which $\|U_{1}^{T}U_{1}-I_{4}\|_{2}\approx1.2\times10^{-15}$,
$\|U_{2}^{T}U_{2}-I_{4}\|_{2}\approx1.4\times10^{-15}$, $\|V_{1}^{T}V_{1}-I_{4}\|_{2}\approx5.8\times10^{-16}$,
$ $$\|V_{2}^{T}V_{2}-I_{4}\|_{2}\approx1.7\times10^{-15}$, $\|E_{11}\|_{2}\approx1.3\times10^{-12}$,
$\|E_{12}\|_{2}\approx6.1\times10^{-13}$, $\|E_{21}\|_{2}\approx1.6\times10^{-12}$,
and $\|E_{22}\|_{2}\approx9.3\times10^{-13}$. The algorithm performs
stably.

\paragraph{Haar measure.}

Let $X$ be a random $40$-by-$40$ orthogonal matrix from Haar measure,
let $p=18$ and $q=15$, and compute the CS decomposition of $X$
using $\mathbf{csd}$. Actually, it is impossible to sample exactly
from Haar measure in floating-point, so define $X$ by the following
Matlab code. \lstset{escapechar=\%,mathescape}
\begin{lstlisting}
[X,R] = qr(randn(40));
X = X * diag(sign(randn(40,1)));
\end{lstlisting} Over 1000 trials of our implementation, $\|U_{1}^{T}U_{1}-I_{18}\|_{2}$,
$\|U_{2}^{T}U_{2}-I_{22}\|_{2}$, $\|V_{1}^{T}V_{1}-I_{15}\|_{2}$,
$\|V_{2}^{T}V_{2}-I_{25}\|_{2}$, $\|E_{11}\|_{2}$, $\|E_{12}\|_{2}$,
$\|E_{21}\|_{2}$, and $\|E_{22}\|_{2}$ were all less than $2\varepsilon$,
where $\varepsilon$ was defined to be the greater of ten times machine
epsilon or $\|X^{T}X-I_{40}\|_{2}$.

\paragraph{Clusters of singular values.}

Let $\delta_{1},\delta_{2},\delta_{3},\dots,\delta_{21}$ be independent
and identically distributed random variables each with the same distribution
as $10^{-18U(0,1)}$, in which $U(0,1)$ is a random variable uniformly
distributed on $[0,1]$. For $i=1,\dots,20$, let $\theta_{i}=\frac{\pi}{2}\cdot\frac{\sum_{k=1}^{i}\delta_{k}}{\sum_{k=1}^{21}\delta_{k}}$,
and let $C=\diag(\cos(\theta_{1}),\dots,\cos(\theta_{20}))$, $S=\diag(\sin(\theta_{1}),\dots,\sin(\theta_{20}))$,
and \[
X=\left[\begin{array}{cc}
U_{1}\\
 & U_{2}\end{array}\right]\left[\begin{array}{rc}
C & S\\
-S & C\end{array}\right]\left[\begin{array}{cc}
V_{1}\\
 & V_{2}\end{array}\right]^{T},\]
in which $U_{1}$, $U_{2}$, $V_{1}$, and $V_{2}$ are random orthogonal
matrices from Haar measure. (These matrices can be sampled as $X$
was sampled in the previous test case.) The random matrix $X$ is
designed so that $C$ and $S$ have clustered singular values, as
well as singular values close to $0$ and $1$. Such singular values
break the naive CSD algorithm. Compute the CSD of $X$ with $p=q=20$
using $\mathbf{csd}$. Over 1000 trials of our implementation, $\|U_{1}^{T}U_{1}-I_{18}\|_{2}$,
$\|U_{2}^{T}U_{2}-I_{22}\|_{2}$, $\|V_{1}^{T}V_{1}-I_{15}\|_{2}$,
$\|V_{2}^{T}V_{2}-I_{25}\|_{2}$, $\|E_{11}\|_{2}$, $\|E_{12}\|_{2}$,
$\|E_{21}\|_{2}$, and $\|E_{22}\|_{2}$ were all less than $3\varepsilon$,
with $\varepsilon$ defined as in the previous test case.

\paragraph{$\theta^{(0)}$ and $\phi^{(0)}$ chosen uniformly from $\left[0,\frac{\pi}{2}\right]$}

Choose $\theta_{1},\dots,\theta_{20}$ and $\phi_{1},\dots,\phi_{19}$
independently and uniformly from the interval $\left[0,\frac{\pi}{2}\right]$,
and let \[
X=\left[\begin{array}{cc}
B_{11}(\theta,\phi) & B_{12}(\theta,\phi)\\
B_{21}(\theta,\phi) & B_{22}(\theta,\phi)\end{array}\right].\]
Compute the CSD of $X$ with $p=q=20$ using $\mathbf{csd}$. Over
1000 trials, $\|U_{1}^{T}U_{1}-I_{18}\|_{2}$, $\|U_{2}^{T}U_{2}-I_{22}\|_{2}$,
$\|V_{1}^{T}V_{1}-I_{15}\|_{2}$, $\|V_{2}^{T}V_{2}-I_{25}\|_{2}$,
$\|E_{11}\|_{2}$, $\|E_{12}\|_{2}$, $\|E_{21}\|_{2}$, and $\|E_{22}\|_{2}$
were all less than $4\varepsilon$, with $\varepsilon$ defined as
above.

\paragraph{$\theta^{(0)}$ and $\phi^{(0)}$ chosen randomly from $\left\{ 0,\frac{\pi}{4},\frac{\pi}{2}\right\} $.}

Repeat the previous test case, but with $\theta_{1},\dots,\theta_{20}$
and $\phi_{1},\dots,\phi_{19}$ chosen uniformly from the three-element
set $\left\{ 0,\frac{\pi}{4},\frac{\pi}{2}\right\} $. This produces
test matrices with many zeros, which can tax the novel aspects of
our extension of the Golub-Kahan-Demmel SVD step. Over 1000 trials,
$\|U_{1}^{T}U_{1}-I_{18}\|_{2}$, $\|U_{2}^{T}U_{2}-I_{22}\|_{2}$,
$\|V_{1}^{T}V_{1}-I_{15}\|_{2}$, $\|V_{2}^{T}V_{2}-I_{25}\|_{2}$,
$\|E_{11}\|_{2}$, $\|E_{12}\|_{2}$, $\|E_{21}\|_{2}$, and $\|E_{22}\|_{2}$
were all less than $\varepsilon$, with $\varepsilon$ defined as
above.

\begin{acknowledgment*}
The thoughtful comments of an anonymous referee are appreciated.
\end{acknowledgment*}
\appendix

\section{\label{sec:proofs}Additional proofs}

\subsection{Proof of Lemma \ref{lem:svdstepequivalenttoqrsteponbtb} }

\begin{proof}
The proof is by induction on $i$.

That $T_{1}=G_{1}$ is straightforward to prove.

Assume by induction that $T_{j}=G_{j}$, $j=1,\dots,i-1$. At line
\ref{line:computeT} of Algorithm \ref{alg:svdstep}, $\bar{B}=S_{i-1}^{T}\cdots S_{1}^{T}BT_{1}\cdots T_{i-1}$,
and so \[
\bar{B}^{T}\bar{B}-\sigma^{2}I=T_{i-1}^{T}\cdots T_{1}^{T}(B^{T}B-\sigma^{2}I)T_{1}\cdots T_{i-1}.\]
At line \ref{line:computeG} of Algorithm \ref{alg:implicitqrtridiagonal},
\[
\bar{A}=G_{i-1}^{T}\cdots G_{1}^{T}(A-\lambda I)G_{1}\cdots G_{i-1}.\]
By the induction hypothesis and the fact that $B^{T}B-\sigma^{2}I=A-\lambda I$,
we have (at the current step), \begin{equation}
\bar{B}^{T}\bar{B}-\sigma^{2}I=\bar{A}.\label{eq:appendixeq1}\end{equation}
We show $T_{i}=G_{i}$ up to signs using three cases.

\emph{Case 1: $\bar{A}(i:i+1,i-1)$ is nonzero.} From (\ref{eq:appendixeq1}),
$\bar{A}(i:i+1,i-1)=\bar{B}(i-1,i-1)\bar{B}(i-1,i:i+1)^{T}$. Hence,
$\bar{B}(i-1,i:i+1)^{T}$ is nonzero, $\bar{A}(i:i+1,i-1)$ is colinear
with $\bar{B}(i-1,i:i+1)^{T}$, and $T_{i}=G_{i}$.

\emph{Case 2: $\bar{A}(i:i+1,i-1)$ and $\bar{B}(i-1,i:i+1)^{T}$
equal the zero vector.} Then the Givens rotations $T_{i}$ and $G_{i}$
are based on $\bar{A}(i:i+1,i)$ and $(\|\bar{B}(:,i)\|^{2}-\sigma^{2},\bar{B}(:,i)^{T}\bar{B}(:,i+1))^{T}$,
respectively, which are equal according to (\ref{eq:appendixeq1}).
Hence, $T_{i}=G_{i}$.

\emph{Case 3: $\bar{A}(i:i+1,i-1)$ is the zero vector but $\bar{B}(i-1,i:i+1)$
is nonzero. }Then $\bar{B}$ must have the form\[
\bar{B}=\left[\begin{array}{cccccc}
* & *\\
 & * & *\\
 &  & 0 & a & b\\
 &  &  & c & d\\
 &  &  &  & * & *\\
 &  &  &  &  & *\end{array}\right]\]
with $a$ and $b$ not both zero. Let's roll back the previous Givens
rotation:\[
(S_{i-1}^{T})^{-1}\bar{B}=S_{i-1}\bar{B}=\left[\begin{array}{cccccc}
* & *\\
 & * & *\\
 &  & 0 & x\\
 &  & 0 & y & z\\
 &  &  &  & * & *\\
 &  &  &  &  & *\end{array}\right].\]
Hence,\[
\bar{A}=\bar{B}^{T}\bar{B}-\sigma^{2}I=(S_{i-1}\bar{B})^{T}(S_{i-1}\bar{B})-\sigma^{2}I=\left[\begin{array}{cccccc}
* & *\\
* & * & *\\
 & * & * & 0 & 0\\
 &  & 0 & x^{2}+y^{2}-\sigma^{2} & yz\\
 &  & 0 & yz & * & *\\
 &  &  &  & * & *\end{array}\right].\]
Now, $S_{i-1}=\mathbf{givens}(i-1,i,(x^{2}-\sigma^{2},xy)^{T})=\pm\frac{1}{r}\left[\begin{array}{cc}
x^{2}-\sigma^{2} & -xy\\
xy & x^{2}-\sigma^{2}\end{array}\right]$, in which $r=\|(x^{2}-\sigma^{2},xy)^{T}\|$, so\[
\bar{B}=S_{i-1}^{T}S_{i-1}\bar{B}=\left[\begin{array}{cccccc}
* & *\\
 & * & *\\
 &  & 0 & \pm\frac{x}{r}(x^{2}+y^{2}-\sigma^{2}) & \pm\frac{x}{r}yz\\
 &  & 0 & \pm\frac{y}{r}(-\sigma^{2}) & \pm\frac{1}{r}(x^{2}-\sigma^{2})z\\
 &  &  &  & * & *\\
 &  &  &  &  & *\end{array}\right].\]
By assumption, $\bar{B}(i-1,i:i+1)^{T}$ is nonzero, and hence by
inspection $\bar{B}(i-1,i:i+1)^{T}$ and $\bar{A}(i:i+1,i)$ are colinear.
Therefore, $T_{i}=G_{i}$. $\qed$
\end{proof}
{}

\subsection{Proof of Lemma \ref{lem:svdstepequivalenttoqrsteponbbt}}

\begin{proof}
Assume by induction that $S_{j}=G_{j}$ for $j=1,\dots,i-1$. At line
\ref{line:computeS} of Algorithm \ref{alg:svdstep}, $\bar{B}=S_{i-1}^{T}\cdots S_{1}^{T}BT_{1}\cdots T_{i}$,
and so \[
\bar{B}\bar{B}^{T}-\sigma^{2}I=S_{i-1}^{T}\cdots S_{1}^{T}(BB^{T}-\sigma^{2}I)S_{1}\cdots S_{i-1}.\]
At line \ref{line:computeG} of Algorithm \ref{alg:implicitqrtridiagonal},
\[
\bar{A}=G_{i-1}^{T}\cdots G_{1}^{T}(A-\lambda I)G_{1}\cdots G_{i-1}.\]
By the induction hypothesis and the fact that $BB^{T}-\sigma^{2}I=A-\lambda I$,
we have (at the current step), \begin{equation}
\bar{B}\bar{B}^{T}-\sigma^{2}I=\bar{A}.\label{eq:appendixeq2}\end{equation}
We show $S_{i}=G_{i}$ using three cases.

\emph{Case 1: $\bar{A}(i:i+1,i-1)$ is nonzero.} From (\ref{eq:appendixeq2}),
$\bar{A}(i:i+1,i-1)=\bar{B}(i-1,i)\bar{B}(i:i+1,i)$. Hence, $\bar{B}(i:i+1,i)$
is nonzero, $\bar{A}(i:i+1,i-1)$ is colinear with $\bar{B}(i:i+1,i)$,
and $S_{i}=G_{i}$.

\emph{Case 2: $\bar{A}(i:i+1,i-1)$ and $\bar{B}(i:i+1,i)$ equal
the zero vector.} Then the Givens rotations $G_{i}$ and $S_{i}$
are based on $\bar{A}(i:i+1,i)$ and $(\|\bar{B}(i,:)\|^{2}-\sigma^{2},\bar{B}(i,:)\bar{B}(i+1,:))$,
respectively, which are equal according to (\ref{eq:appendixeq2}).
Hence, $S_{i}=G_{i}$.

\emph{Case 3: $\bar{A}(i:i+1,i-1)$ is the zero vector but $\bar{B}(i:i+1,i)$
is nonzero. }Then $\bar{B}$ must have the form\[
\bar{B}=\left[\begin{array}{cccccc}
* & *\\
 & * & 0\\
 &  & a & c\\
 &  & b & d & *\\
 &  &  &  & * & *\\
 &  &  &  &  & *\end{array}\right]\]
with $a$ and $b$ not both zero. Let's roll back the previous Givens
rotation:\[
\bar{B}T_{i}^{-1}=\bar{B}T_{i}^{T}=\left[\begin{array}{cccccc}
* & *\\
 & * & 0 & 0\\
 &  & x & y\\
 &  &  & z & *\\
 &  &  &  & * & *\\
 &  &  &  &  & *\end{array}\right].\]
Hence,\[
\bar{A}=\bar{B}\bar{B}^{T}-\sigma^{2}I=(\bar{B}T_{i}^{T})(\bar{B}T_{i}^{T})^{T}-\sigma^{2}I=\left[\begin{array}{cccccc}
* & *\\
* & * & 0 & 0\\
 & 0 & x^{2}+y^{2}-\sigma^{2} & yz\\
 & 0 & yz & * & *\\
 &  &  & * & * & *\\
 &  &  &  & * & *\end{array}\right].\]
Now, $T_{i}=\mathbf{givens}(i,i+1,(x^{2}-\sigma^{2},xy)^{T})=\pm\frac{1}{r}\left[\begin{array}{cc}
x^{2}-\sigma^{2} & -xy\\
xy & x^{2}-\sigma^{2}\end{array}\right]$, in which $r=\|(x^{2}-\sigma^{2},xy)^{T}\|$, so\[
\bar{B}=\bar{B}T_{i}^{T}T_{i}=\left[\begin{array}{cccccc}
* & *\\
 & * & 0 & 0\\
 &  & \pm\frac{x}{r}(x^{2}+y^{2}-\sigma^{2}) & \pm\frac{y}{r}(-\sigma^{2})\\
 &  & \pm\frac{x}{r}yz & \pm\frac{1}{r}(x^{2}-\sigma^{2})z & *\\
 &  &  &  & * & *\\
 &  &  &  &  & *\end{array}\right].\]
By assumption, $\bar{B}(i:i+1,i)$ is nonzero, and hence by inspection
$\bar{B}(i:i+1,i)$ and $\bar{A}(i:i+1,i)$ are colinear. Therefore,
$S_{i}=G_{i}$. $\qed$
\end{proof}
\bibliographystyle{plain}
\bibliography{sutton-csd}

\end{document}